\documentclass[10pt,a4paper]{article}
\usepackage{amsthm,amssymb,mathrsfs,setspace,amsmath}
\usepackage{graphicx}
\usepackage{lipsum}
\usepackage{caption}
\usepackage{multicol}
\usepackage{float}
\usepackage[a4paper,bindingoffset=0.2in,
left=0.5in,right=0.5in,
top=1in,bottom=1in,footskip=.25in]{geometry}
\title{Linear Canonical Stockwell Transform and the associated Multiresolution Analysis}
\author{Bivek Gupta$^a$\thanks{$^a$bivekgupta040792@gmail.com}, Amit K. Verma$^b$\thanks{$^b$akverma@iitp.ac.in}, \\\small{\it{$^{a,b}$ Department of Mathematics,}} \\\small{\it{Indian Institute of Technology Patna,}}\\\small{\it{ Bihta, Patna 801103, (BR) India.}}}
\theoremstyle{definition}
\newtheorem{definition}{Definition}[section]

\newtheorem{proposition}{Proposition}[section]
\newtheorem{theorem}{Theorem}[section]
\newtheorem{corollary}{Corollary}[section]
\newtheorem{remark}{Remark}[section]
\begin{document}
\maketitle
\begin{abstract}
In this article, we give a new definition of the linear canonical Stockwell transform (LCST) and study its basic properties along with the inner product relation, reconstruction formula and also characterize the range of the transform and show that its range is the reproducing kernel Hilbert space. We also develop a multiresolution analysis (MRA) associated with the proposed transform together with the construction of the orthonormal basis for $L^2(\mathbb{R}).$
\end{abstract}
\section{Introduction}
Although the wavelet transform overcomes the limitation of the short time Fourier transform (STFT) and has served as the powerful tool for signal processing, it lacks phase information. Stockwell et al. \cite{stockwell1996localization},\cite{stockwell2007basis},\cite{du2007continuous} proposed a hybrid integral transform by combining the advantages of both the STFT and the wavelet transform, to overcome this limitation, called the  Stockwell transform which is defined for $f\in L^2(\mathbb{R})$ and a given window function $\psi\in L^2(\mathbb{R})$ as 
$$\left(\mathcal{S}_\psi f \right)(a,b)=\frac{1}{\sqrt{2\pi}}\int_{\mathbb{R}}f(t)\overline{\psi_{a,b}(t)}dt,~a\in\mathbb{R}^+,~b\in\mathbb{R},$$ where
$\psi_{a,b}(t)=\mathcal{M}_a\tau_b\mathcal{D}_{a}\psi(t)=ae^{iat}\psi(a(t-b)),$ 
 $(\mathcal{D}_{a}\psi)(t)=a\psi(at),$
$(\mathcal{M}_a\psi)(t)=e^{iat}\psi(t),$
 $(\tau_b\psi)(t)=\psi(t-b).$ 
Equivalently, it can be written as 
$$\left(\mathcal{S}_\psi f \right)(a,b)=\frac{1}{\sqrt{2\pi}}[(\mathcal{M}_{-a}f)\star(D_a\tilde{\psi})](b),$$\\
where $\tilde{\psi}(t)=\overline{\psi}(-t)$ and $\star$ denotes the classical convolution defined by
$$(f\star g)(t)=\int_{\mathbb{R}}f(x)g(t-x)dx.$$

Even though the Stockwell transform
is a promising tool in the analysis of the non-stationary signals, it is incompetent for the analysis of those signals whose energy is not well concentrated in the time-frequency plane, for example the chirp like signal which is ubiquitous in nature \cite{dai2017new}. 

As a generalization of Fourier transform (FT)\cite{debnath2014integral} and fractional Fourier transform (FrFT) \cite{almeida1994fractional,verma2022note}, linear canonical transform (LCT) is a four parameter family of linear integral transform proposed by Mohinsky and Quesne \cite{moshinsky1971linear} and is considered as the important tool for non-stationary signal processing. Because of the extra degrees of freedom, as compared to the FT and FrFT, its application can be found in a number of fields including signal separation \cite{sharma2006signal}, signal reconstruction \cite{wei2014reconstruction}, filter designing \cite{barshan1997optimal} and many more. For more detail on LCT and its application we refer the reader to \cite{healy2015linear}. 

For the signal whose energy is not well concentrated in the frequency domain, LCT is an appropriate tool. However, because of its global kernel it is inadequate to indicate the time localization of the LCT spectral components, and thus it is improper for processing the signals whose LCT spectral features varies with time. The window linear canonical transform (WLCT)\cite{kou2012windowed} and linear canonical wavelet transform (LCWT)\cite{gupta2022linear}
are thus proposed to address this issue. 


Replacing the classical convolution $\star$ by the fractional convolution $\star_{\alpha},$ i.e.,
$$(f\star_{\alpha}g)(x)=\int_{\mathbb{R}}f(y)g(x-y)e^{\frac{i}{2}(y^2-x^2)\cot\alpha}dy,$$
Srivastava et al.\cite{srivastava2020family} introduced the generalized Stockwell transform and studied its basis properties and the uncertainty principles associated with it. Further replacing the classical convolution by the linear canonical convolution \cite{deng2006convolution} $\star_{A},$ i.e.,
$$(f\star_{M}g)(x)=\int_{\mathbb{R}}f(y)g(x-y)e^{-\frac{iAy(x-y)}{B}}dy,$$ where $M$ is the matrix parameter $(A,B;C,D),$ such that $B\neq 0$ and $AD-BC=1,$
Shah et al.\cite{shah2020linear} defined the linear canonical Stockwell transform (LCST). They studied the orthogonality relation, inversion formula and the range theorem. The discrete version of this transform along with its reconstruction formula is also obtained. They discussed the time-LCT domain frequency resolution and the constant $Q-$property associated with the LCST.

Multiresolution analysis (MRA) has become an important tool for the construction of a discrete wavelet system  i.e., orthonormal basis for $L^2(\mathbb{R}).$ It was first formulated by Mallet \cite{mallat1989multiresolution}. Some of the important wavelets that have been obtained from the MRA are Haar wavelet, Shannon wavelet, Meyer wavelet, Daubechies wavelet, etc. \cite{daubechies1992ten,debnath2002wavelet}. Not all wavelets are derived via MRA, there are some wavelets, like Mexican hat wavelet and Morlet wavelet, that does not come from an MRA (see\cite{pereyra2012harmonic}). Shi et al.\cite{shi2015multiresolution} developed the MRA associated to the FrWT \cite{shi2012novel} and constructed some fractional wavelets. Motivated by \cite{hernandez1996first}, Ahmad et al. \cite{ahmad2021fractional} proved that the conditions in the definition of fractional MRA \cite{shi2015multiresolution} are not independent and also they characterized the scaling function associated with the fractional MRA. Dai et al.\cite{dai2017new} introduced a novel FrWT and studied the  associated MRA followed by the construction of orthogonal wavelets. Guo et al.\cite{guo2021novel} introduced a novel FrWT with one more extra degree of freedom as compared to the one studied in \cite{dai2017new} and studied the associated MRA followed by the construction of orthogonal wavelets. Recently, Shah et al. \cite{shah2022special} introduced the special affine wavelet transform based on the special affine convolution given in \cite{xiang2014convolution}. They introduced the special affine MRA and constructed some discrete orthonormal special affine wavelets.

So far in the literature, we have not seen any paper on the MRA that is associated with the Stockwell transform \cite{du2007continuous}, fractional Stockwell transform \cite{srivastava2020family} or linear canonical Stockwell transform \cite{shah2020linear}. This paper deals with the novel way of defining an MRA followed by the construction of orthonormal  basis from the newly defined MRA.

The main objectives of this paper are 
(i) to introduce a  new time-frequency analysing tool and study some of its basic properties along with the inner product relation, reproducing formula and the characterization of the range of the LCST. Our definition is more general than the one given in \cite{shah2020linear},
(ii) to study the time frequency analysis and the associated constant $Q-$factor,
(iii) to define the MRA associated with the novel LCST and construct orthonormal basis for $L^2(\mathbb{R}).$

The paper is arranged as follows. In section 2, we recall the definition of LCT and some of its properties. In section 3, we define novel LCST and study some of its basic properties including inner product relation, reconstruction formula and also characterize the range of the LCST. The MRA associated with the LCST followed by the construction of orthonormal basis for $L^2(\mathbb{R})$ are discussed in  section 4. Finally, in section 5, we conclude our paper.

\section{Preliminaries}
\subsection{Linear Canonical Transform}
We briefly recall the definition of LCT and its important properties that will be used in this paper.
\begin{definition}
The LCT of $f\in L^2(\mathbb{R})$, with respect to a matrix parameter $$M=
\begin{bmatrix}
A & B\\
C & D
\end{bmatrix}, AD-BC=1$$
is defined as 
$$(\mathcal{L}^{M}f)(\xi)=\begin{cases}
\displaystyle\int_{\mathbb{R}}f(t)K_{M}(t,\xi)d t, B\neq 0,\\
\sqrt{D}e^{\frac{i}{2}CD\xi^2}F(D\xi),~B=0,
\end{cases}$$
where the kernel $K_{M}(t,\xi)$ is given by
\begin{equation}\label{P4KernelLCT}
K_{M}(t,\xi)=
 \frac{1}{\sqrt{2\pi iB}} e^{\frac{i}{2}\left(\frac{A}{B}t^2-\frac{2}{B}\xi t+\frac{D}{B}\xi^2\right)},~\xi\in\mathbb{R}.
\end{equation}
\end{definition}
Among several important properties of the LCT, the important among them that will be used in this paper, is the Parseval's formula
\begin{equation}\label{P4ParsevalLCT}
\int_{\mathbb{R}}f(t)\overline{g(t)}dt=\int_{\mathbb{R}}\left(\mathcal{L}^Mf\right)(\xi)\overline{\left(\mathcal{L}^Mg\right)(\xi)}d\xi,~\mbox{where}~f,~g\in L^2(\mathbb{R}).
\end{equation}\label{P4PlancherelLCT}
Particularly, if $f=g,$ then we have the Plancherel's formula for the LCT
\begin{equation}\|f\|_{L^2(\mathbb{R})}=\left\|\mathcal{L}^Mf\right\|_{L^2(\mathbb{R})}.
\end{equation}
The LCTs satisfies the additive property, i.e.,
\begin{equation}\label{P4SemigrouplLCT}
\mathcal{L}^M\mathcal{L}^Nf=\mathcal{L}^{MN}f,~\mbox{where} ~f\in L^2(\mathbb{R}),
\end{equation}
and the inversion property
\begin{equation}\label{P4InverseLCT}
\mathcal{L}^{M^{-1}}\left(\mathcal{L}^{M}f\right)=f,
\end{equation}
where, $M^{-1}$ denotes the inverse of $M.$
\section{Linear Canonical Stockwell Transform}
We propose a new integral transform namely the LCST. We shall discuss some of its basic properties along with the inner product relation, reconstruction formula and also prove that the range of the transform is a reproducing kernel Hilbert space. For the notational convenience, we write the matrix $M=
\begin{bmatrix}
A & B\\
C & D
\end{bmatrix}$ as $M=(A,B;C,D).$

\begin{definition}\label{P4DefinitionLCST}
Let $f\in L^2(\mathbb{R}),$ $M_{j}=(A_{j},B_{j};C_{j},D_{j}),~j=1,2$ be matrices with $A_{j}D_{j}-B_{j}C_{j}=1~\mbox{and}~B_{j}\neq 0,~j=1,2$ then the LCST of $f$  with respect to $M_{j}$ and $\psi\in L^2(\mathbb{R})$ is defined by
$$\left(S^{M_{1},M_{2}}_{\psi}f\right)(a,b)=e^{-\frac{iA_{1}}{2B_{1}}b^2}\left\{\tilde{f}(t)e^{\frac{iA_{1}}{2B_{1}}t^2}\star\overline{a\psi(-at)e^{\frac{iA_{2}}{2B_{2}}(at)^2}}\right\}(b),~a\in\mathbb{R^+},b\in\mathbb{R},$$
where $\tilde{f}(t)=e^{-i\frac{at}{B_{1}}}f(t)$ and $\star$ denotes the classical convolution.
\end{definition}
Equivalently, 
\begin{align*}
\left(S^{M_{1},M_{2}}_\psi f\right)(a,b)
&=e^{-\frac{iA_{1}}{2B_{1}}b^2}\int_{\mathbb{R}}e^{-i\frac{at}{B_{1}}}f(t)e^{\frac{iA_{1}}{2B_{1}}t^2}\overline{a\psi(-a(b-t))e^{\frac{iA_{2}}{2B_{2}}\left\{a(t-b)\right\}^2}}dt\\
&=\int_{\mathbb{R}}f(t)\overline{e^{-\frac{iA_{1}}{2B_{1}}(t^2-b^2)+\frac{iA_{2}}{2B_{2}}\left\{a(t-b)\right\}^2+\frac{iat}{B_{1}}}a\psi(a(t-b))}dt\\
&=\int_{\mathbb{R}}f(t)\overline{\psi^{M_{1},M_{2}}_{a,b}(t)}dt,
\end{align*}
where, with $\psi_{a,b}(t)=a\psi\left(a(t-b)\right),$ 
\begin{equation}\label{P4KernelLCST}
\psi^{M_{1},M_{2}}_{a,b}(t)=e^{-\frac{iA_{1}}{2B_{1}}(t^2-b^2)+\frac{iA_{2}}{2B_{2}}\left\{a(t-b)\right\}^2+\frac{iat}{B_{1}}}\psi_{a,b}(t).
\end{equation}
Thus we have an equivalent definition of the LCST  as 
\begin{equation}\label{P4EquivalentLCST}
\left(S^{M_{1},M_{2}}_{\psi}f\right)(a,b)=\left\langle f,\psi^{M_{1},M_{2}}_{a,b} \right\rangle_{L^2(\mathbb{R})}.
\end{equation}

It is to be noted that depending on the different choice of the matrix $M_1$ and $M_2$ we have different family of novel integral transforms:
\begin{itemize}
\item For $M_1=(\cos\alpha,\sin\alpha;-\sin\alpha,\cos\alpha),~M_2=(\cos\beta,\sin\beta;-\sin\beta,\cos\beta),~\alpha,\beta\neq n\pi~(n\in\mathbb{Z}),$ we obtain the novel fractional Stockwell transform (FrST)
\begin{align*}
\left(S^{\alpha,\beta}_\psi f\right)(a,b)=a\int_{\mathbb{R}}f(t)\overline{\psi\left(a(t-b)\right)}e^{\frac{i}{2}(t^2-b^2)\cot\alpha-\frac{i}{2}\left\{a(t-b)\right\}^2\cot\beta-iat\csc\alpha}dt.
\end{align*}
\item For $M_1=(1,B_1;0,1),~M_2=(1,B_2;0,1),~B_1,B_2\neq 0$ we obtain the novel Fresnel Stockwell transform
\begin{align*}
\left(S^{B_1,B_2}_\psi f\right)(a,b)=a\int_{\mathbb{R}}f(t)\overline{\psi\left(a(t-b)\right)}e^{\frac{iA_{1}}{2B_{1}}(t^2-b^2)-\frac{iA_{2}}{2B_{2}}\left\{a(t-b)\right\}^2-\frac{iat}{B_{1}}}dt.
\end{align*}
\item For $M_1=M_2=(0,1;-1,0)$ we obtain the classical Stockwell transform (ST)
\begin{align*}
\left(S_\psi f\right)(a,b)=a\int_{\mathbb{R}}f(t)\overline{\psi\left(a(t-b)\right)}e^{-iat}dt.
\end{align*}
\end{itemize}
We now establish a fundamental relation between the LCST and the LCT. This relation will be useful in obtaining the resolution of time and linear canonical spectrum in the time-LCT-frequency plane and inner product relation associated with the LCST. 
\begin{proposition}
If $S_{\psi}^{M_{1},M_{2}}f$ and $\mathcal{L}^{M_{2}}f$ are respectively the LCST and the LCT of $f\in L^2(\mathbb{R}).$ Then,
\begin{equation}\label{P4FundamentalRelationLCST-LCT}
\left(S_{\psi}^{M_{1},M_{2}}f\right)(a,b)=\overline{\sqrt{2\pi iB_{2}}}e^{-\frac{iab}{B_{1}}}\int_{\mathbb{R}}e^{\frac{iD_{2}}{2B_{2}}\left(\frac{B_{2}\xi}{B_{1}a}\right)^2}\left(\mathcal{L}^{M_{1}}f\right)(\xi)\overline{\left(\mathcal{L}^{M_{2}}\left\{e^{\frac{it}{B_{1}}}\psi(t)\right\}\right)\left(\frac{B_{2}\xi}{B_{1}a}\right)}K_{M_{1}^{-1}}(b,\xi)d\xi.
\end{equation}
\end{proposition}
\begin{proof}
Form the definition of the LCT and $\psi_{a,b}^{M_{1},M_{2}},$ we have
\begin{align*}
\left(\mathcal{L}^{M_{1}}\psi_{a,b}^{M_{1},M_{2}}\right)(\xi)
&=\int_{\mathbb{R}}a\psi(a(t-b))e^{-\frac{iA_{1}}{2B_{1}}(t^2-b^2)+\frac{iA_{2}}{2B_{2}}\left\{a(t-b)\right\}^2+\frac{iat}{B_{1}}}\frac{1}{\sqrt{2\pi iB_{1}}}e^{\frac{i}{2}\left(\frac{A_{1}}{B_{1}}t^2-\frac{2}{B_{1}}\xi t+\frac{D_{1}}{B_{1}}\xi^2\right)}dt\\
&=\frac{1}{\sqrt{2\pi iB_{1}}} \int_{\mathbb{R}}a\psi(at)e^{\frac{iA_{1}}{2B_{1}}b^2+\frac{iA_{2}}{2B_{2}}(at)^2+\frac{ia(t+b)}{B_{1}}}e^{\frac{i}{2}\left\{-\frac{2}{B_{1}}\xi (t+b)+\frac{D_{1}}{B_{1}}\xi^2\right\}}dt\\
&=\int_{\mathbb{R}}a\psi(at)e^{\frac{iA_{2}}{2B_{2}}(at)^2+\frac{ia(t+b)}{B_{1}}}e^{\frac{i}{2}\left(-\frac{2}{B_{1}}\xi t\right)}K_{M_{1}}(b,\xi)dt\\
&=K_{M_{1}}(b,\xi)\int_{\mathbb{R}}\psi(t)e^{\frac{iA_{2}}{2B_{2}}t^2+\frac{i(t+ab)}{B_{1}}}e^{\frac{i}{2}\left\{-\frac{2}{B_{1}}\left(\frac{\xi}{a}\right)t\right\}}dt\\
&=e^{\frac{iab}{B_{1}}-\frac{iD_{2}}{2B_{2}}\left(\frac{B_{2}\xi}{B_{1}a}\right)^2}K_{M_{1}}(b,\xi)\int_{\mathbb{R}}e^{\frac{it}{B_{1}}}\psi(t)\sqrt{2\pi iB_{2}}K_{M_{2}}\left(t,\frac{B_{2}\xi}{B_{1}a}\right)dt.
\end{align*}
Therefore, we have
\begin{equation}\label{P4LCT-Kernel-LCST}
\left(\mathcal{L}^{M_{1}}\psi_{a,b}^{M_{1},M_{2}}\right)(\xi)=\sqrt{2\pi iB_{2}}e^{\frac{iab}{B_{1}}-\frac{iD_{2}}{2B_{2}}\left(\frac{B_{2}\xi}{B_{1}a}\right)^2}K_{M_{1}}(b,\xi)\left(\mathcal{L}^{M_{2}}\left\{e^{\frac{it}{B_{1}}}\psi(t)\right\}\right)\left(\frac{B_{2}\xi}{B_{1}a}\right).
\end{equation}
Using \eqref{P4ParsevalLCT} in \eqref{P4EquivalentLCST}, we get 
$$\left(S_{\psi}^{M_{1},M_{2}}f\right)(a,b)=\left\langle \mathcal{L}^{M_{1}}f,\mathcal{L}^{M_{1}}\left(\psi_{a,b}^{M_{1},M_{2}}\right)\right\rangle_{L^2(\mathbb{R})}.$$
Using \eqref{P4LCT-Kernel-LCST}, we have
\begin{equation*}
\left(S_{\psi}^{M_{1},M_{2}}f\right)(a,b)=\overline{\sqrt{2\pi iB_{2}}}e^{-\frac{iab}{B_{1}}}\int_{\mathbb{R}}e^{\frac{iD_{2}}{2B_{2}}\left(\frac{B_{2}\xi}{B_{1}a}\right)^2}\left(\mathcal{L}^{M_{1}}f\right)(\xi)\overline{\left(\mathcal{L}^{M_{2}}\left\{e^{\frac{it}{B_{1}}}\psi(t)\right\}\right)\left(\frac{B_{2}\xi}{B_{1}a}\right)}K_{M_{1}^{-1}}(b,\xi)d\xi.
\end{equation*}

This completes the proof.
\end{proof}
\subsection{Time-LCT frequency analysis}
From the equation
$$\left(S^{M_{1},M_{2}}_\psi f\right)(a,b)=\int_\mathbb{R}f(t)\overline{\psi^{M_{1},M_{2}}_{a,b}(t)}dt,$$ it follows that if $\psi^{M_{1},M_{2}}_{a,b}$ is supported in the time domain, then so is $\left(S^{M_{1},M_{2}}_\psi f\right)(a,b).$

Also, from equation \eqref{P4FundamentalRelationLCST-LCT}, 
it follows that the LCST can provide the local property of $f(t)$ in the linear canonical domain. Thus the LCST is capable of producing the time and linear canonical domain spectral information simultaneously and represents the signal in the time-LCT frequency domain. More precisely, suppose that $E_\psi$ and $\Delta_{\psi}$ are respectively the center and radius of the window function $\psi$. Then the center and radius of $\psi^{M_{1},M_{2}}_{a,b}$ are given respectively by

$$E\left[\psi^{M_{1},M_{2}}_{a,b}\right]=\frac{1}{a}E_\psi +b,$$ 
and
$$\Delta\left[\psi^{M_{1},M_{2}}_{a,b}\right]=\frac{1}{a}\Delta_\psi.$$

Let $\Psi(\xi)=\left(\mathcal{L}^{M_{2}}\left\{e^{\frac{it}{B_{1}}}\psi(t)\right\}\right)(\xi)$ be a window function with $E_{\Psi}$ and $\Delta_{\Psi}$ as the center and radius respectively in the linear canonical domain . Then for $B_1B_2>0,$ the  center of $\left(\mathcal{L}^{M_{2}}\left\{e^{\frac{it}{B_{1}}}\psi(t)\right\}\right)\left(\frac{B_{2}\xi}{B_{1}a}\right)$ is
$$E\left[\left(\mathcal{L}^{M_{2}}\left\{e^{\frac{it}{B_{1}}}\psi(t)\right\}\right)\left(\frac{B_{2}\xi}{B_{1}a}\right)\right]=\frac{B_{1}a}{B_{2}}E_{\Psi},$$
whereas the radius is
$$\Delta\left[\left(\mathcal{L}^{M_{2}}\left\{e^{\frac{it}{B_{1}}}\psi(t)\right\}\right)\left(\frac{B_{2}\xi}{B_{1}a}\right)\right]=\frac{B_{1}a}{B_{2}}\Delta_{\Psi}.$$
Consequently, the window function's $Q$-factor in the linear canonical domain is
$$Q=\frac{\Delta_{\Psi}}{E_{\Psi}},$$
which, for the provided matrices $M_1$ and $M_2,$  is independent of the dilation parameter $a.$ This is called the constant $Q-$property of the LCST.

\subsection{Time-LCT frequency resolution}
The LCST $\left(S_\psi^{M_{1},.M_{2}} f\right)(a,b)$ of $f$ localizes the signal with the time window 
$$\left[\frac{1}{a}E_\psi+b-\frac{1}{a}\Delta_
\psi,\frac{1}{a}E_\psi+b+\frac{1}{a}\Delta_
\psi\right].$$
Similarly, we get that the LCST gives localized information of the linear canonical spectrum of $f$ in the window
$$\left[\frac{B_{1}a}{B_{2}}E_{\Psi}-\frac{B_{1}a}{B_{2}}\Delta_{\Psi},\frac{B_{1}a}{B_{2}}E_{\Psi}+\frac{B_{1}a}{B_{2}}\Delta_{\Psi}\right].$$
Thus, the joint resolution of the LCST in the time and  linear canonical domain is given by a window 
\begin{align}\label{P4Window}
\left[\frac{1}{a}E_\psi+b-\frac{1}{a}\Delta_
\psi,\frac{1}{a}E_\psi+b+\frac{1}{a}\Delta_
\psi\right]\times\left[\frac{B_{1}a}{B_{2}}E_{\Psi}-\frac{B_{1}a}{B_{2}}\Delta_{\Psi},\frac{B_{1}a}{B_{2}}E_{\Psi}+\frac{B_{1}a}{B_{2}}\Delta_{\Psi}\right],
\end{align}
with constant window area $4\frac{B_{1}}{B_{2}}\Delta_{\psi}\Delta_{\Psi}$ in the time-LCT-frequency plane.

Thus, it follows that for a given parameter $M_1,M_2$ the area of \eqref{P4Window} depends on the admissible window $\psi$ and is independent of the parameters $a$ and $b.$ It is to be noted that the window gets narrower for large value of $a$ and wider for small value of $a.$ Thus the window given by the transform is flexible and hence, it is capable of providing the time linear canonical domain information  simultaneously.

We now present some basic properties of LCST.
\begin{theorem}
Let $f,g\in L^2(\mathbb{R}),$  $\alpha,\beta\in\mathbb{C}$, $\lambda> 0$ and $y\in \mathbb{R}.$ Then
\begin{enumerate}
\item $S^{M_{1},M_{2}}_{\psi}(\alpha f+\beta g)=\alpha S^{M_{1},M_{2}}_{\psi}f+\beta S^{M_{1},M_{2}}_{\psi}g$
\item $S^{M_{1},M_{2}}_{\alpha \psi+\beta\phi}f=\bar\alpha S^{M_{1},M_{2}}_{\psi}f+\bar\beta S^{M_{1},M_{2}}_{\phi}g$
\item $\left(S^{M_{1},M_{2}}_{\psi}\delta_{\lambda} f\right)(a,b)=\left(S^{\tilde{M_{1}},M_{2}}_{\psi}f\right)\left(\frac{a}{\lambda},b\lambda\right),~\mbox{where}~ (\delta_\lambda f)(t)=f(\lambda t),~\tilde{M_{1}}=\left(\frac{A_{1}}{\lambda^2},B_{1};C_{1},D_{1}\lambda^2\right)$
\item $\left(S^{M_{1},M_{2}}_{\psi}\tau_{y} f\right)(a,b)=e^{-\frac{iay}{B_{1}}-\frac{iA_{1}}{B_{1}}y(b-y)}\left(S^{\tilde{M_{1}},M_{2}}_{\psi} \left\{e^{\frac{iA_{1}}{B_{1}}yt}f(t)\right\}\right)(a,b-y),~\mbox{where}~ (\tau_{y}f)(t)=f(t-y)$.
\end{enumerate}
\end{theorem}
\begin{proof}
For the brevity of the paper we omit their proofs.
\end{proof}

\begin{definition}
A window function $\psi\in L^2(\mathbb{R})$ is said to be admissible if $C_{\psi,M_{1},M_{2}},$ defined by
\begin{align}\label{P4AdmissibleCondition}
\displaystyle C_{\psi,M_{1},M_{2}}=\int_{\mathbb{R}^+}\left|\left(\mathcal{L}^{M_{2}}\left\{e^{\frac{it}{B_{1}a}}\psi(t)\right\}\right)\left(\frac{B_2\xi}{B_1a}\right)\right|^2\frac{da}{a}
\end{align}
is a constant independent of $\xi$ satisfying $0<C_{\psi,M_{1},M_{2}}<\infty.$
\end{definition}

\begin{theorem}(\textbf{Inner product relation for LCST}).\label{P4InnerProductRelation of LCST}
 Let $\psi$ be an admissible window and $f,g\in L^2(\mathbb{R}),$ then
\begin{equation}\label{P4InnerProductRelationLCST}
\int_{\mathbb{R}^+\times\mathbb{R}}\left(S^{M_{1},M_{2}}_{\psi}f\right)(a,b)\overline{\left(S^{M_{1},M_{2}}_{\psi}g\right)(a,b)}\frac{dadb}{a}=2\pi|B_{1}|C_{\psi,M_{1},M_{2}}\langle f,g\rangle_{L^2(\mathbb{R})},
\end{equation}
where $C_{\psi,M_{1},M_{2}},$ is given by \eqref{P4AdmissibleCondition}.
\end{theorem}
\begin{proof}
Using \eqref{P4FundamentalRelationLCST-LCT}, we get
\begin{align*}
\int_{\mathbb{R}^+\times\mathbb{R}}&\left(S^{M_{1},M_{2}}_{\psi}f\right)(a,b)\overline{\left(S^{M_{1},M_{2}}_{\psi}g\right)(a,b)}\frac{dadb}{a}\\
&=2\pi|B_{1}|\int_{\mathbb{R}^+\times\mathbb{R}}\int_{\mathbb{R}}e^{\frac{iD_{2}}{2B_{2}}\left( \frac{B_{2}\xi}{B_{1}a}\right)^2-\frac{iD_{2}}{2B_{2}}\left(\frac{B_{2}u}{B_{1}a}\right)^2}\left(\mathcal{L}^{M_{1}}f\right)(\xi)\overline{\left(\mathcal{L}^{M_{1}}g\right)(u)} \overline{\left(\mathcal{L}^{M_{2}}\left\{e^{\frac{it}{B_{1}a}}\psi(t)\right\}\right)\left(\frac{B_{2}\xi}{B_{1}a}\right)}\\
&\hspace{6cm}\times\left(\mathcal{L}^{M_{2}}\left\{e^{\frac{it}{B_{1}a}}\psi(t)\right\}\right)\left(\frac{B_{2}u}{B_{1}a}\right)\left(\int_{\mathbb{R}}\overline{K_{M_{2}}(b,\xi)}K_{M_{2}}(b,u)db\right)\frac{da}{a}d\xi du\\
&=2\pi|B_{1}|\int_{\mathbb{R}^+\times\mathbb{R}}\int_{\mathbb{R}}e^{\frac{iD_{2}}{2B_{2}}\left( \frac{B_{2}\xi}{B_{1}a}\right)^2-\frac{iD_{2}}{2B_{2}}\left(\frac{B_{2}u}{B_{1}a}\right)^2}\left(\mathcal{L}^{M_{1}}f\right)(\xi)\overline{\left(\mathcal{L}^{M_{1}}g\right)(u)} \overline{\left(\mathcal{L}^{M_{2}}\left\{e^{\frac{it}{B_{1}a}}\psi(t)\right\}\right)\left(\frac{B_{2}\xi}{B_{1}a}\right)}\\
&\hspace{9cm}\times\left(\mathcal{L}^{M_{2}}\left\{e^{\frac{it}{B_{1}a}}\psi(t)\right\}\right)\left(\frac{B_{2}u}{B_{1}a}\right)\delta(u-\xi)\frac{da}{a}d\xi du\\
&=2\pi|B_{1}|\int_{\mathbb{R}}\left(\mathcal{L}^{M_{1}}f\right)(\xi)\overline{\left(\mathcal{L}^{M_{1}}g\right)(\xi)}\left\{\int_{\mathbb{R}^+}\left|\left(\mathcal{L}^{M_{2}}\left\{e^{\frac{it}{B_{1}a}}\psi(t)\right\}\right)\left(\frac{B_{2}\xi}{B_{1}a}\right)\right|^2\frac{da}{a}\right\}d\xi\\
&=2\pi|B_{1}|C_{\psi,M_{1},M_{2}}\left\langle\mathcal{L}^{M_{1}}f,\mathcal{L}^{M_{1}}g \right\rangle_{L^2(\mathbb{R})}\\
&=2\pi|B_{1}|C_{\psi,M_{1},M_{2}}\langle f,g\rangle_{L^2(\mathbb{R})}
\end{align*}
\end{proof}
\begin{remark}
\begin{itemize}
(Plancherel's theorem for $S^{M_{1},M_{2}}_{\psi}$) Replacing $f=g$ in above equation we have the Plancherel's theorem for $S^{M_{1},M_{2}}_{\psi}$ given by
\begin{equation}\label{P4PlancherelLCST}
\|S^{M_{1},M_{2}}_{\psi} f\|_{L^2\left(\mathbb{R}^+\times\mathbb{R},\frac{dadb}{a}\right)}=2\pi |B_{1}|\sqrt{C_{\psi,M_{1},M_{2}}}\|f\|_{L^2(\mathbb{R})}.
\end{equation}
Thus, it follows that LCST is a bounded linear operator from $L^2(\mathbb{R})$ into $L^2\left(\mathbb{R^+\times\mathbb{R}},\frac{dadb}{a}\right).$ If further $\psi$ is such that $C_{\psi,M_{1},M_{2}}=\frac{1}{2\pi|B_{1}|},$ then the operator is an isometry.
\end{itemize}
\end{remark}
\begin{theorem}(\textbf{Reconstruction formula}).\label{P4ReconstrutionTheoremLCST}
 Let $\psi$ be an admissible window and $f\in L^2(\mathbb{R})$, then $f$ can be reconstructed by the formula
\begin{equation}\label{P4ReconstructionFormulaLCST}
f(t)=\frac{1}{2\pi|B_{1}|C_{\psi,M_{1},M_{2}}}\int_{\mathbb{R}^+\times\mathbb{R}}\left(S^{M_{1},M_{2}}_{\psi}f\right)(a,b)\psi^{M_{1},M_{2}}_{a,b}(t)\frac{dadb}{a},~\mbox{a.e.}
\end{equation}
\end{theorem}
\begin{proof}
From Theorem \ref{P4InnerProductRelation of LCST}, we get
\begin{align*}
2\pi|B_{1}|C_{\psi,M_{1},M_{2}}\langle f,g\rangle_{L^2(\mathbb{R})}
&=\left\langle S^{M_{1},M_{2}}_\psi f,S^{M_{1},M_{2}}_\psi g \right\rangle_{L^2(\mathbb{R^+}\times\mathbb{R},\frac{dadb}{a})}\\
&=\int_{\mathbb{R}^+\times\mathbb{R}}\left(S^{M_{1},M_{2}}_\psi f\right)(a,b)\overline{\left\{\int_{\mathbb{R}}g(t)\overline{\psi^{M_{1},M_{2}}_{a,b}(t)}dt\right\}}\frac{dadb}{a}\\
&=\int_{\mathbb{R}}\left\{\int_{\mathbb{R}^+\times\mathbb{R}}\left(S^{M_{1},M_{2}}_\psi f\right)(a,b)\psi^{M_{1},M_{2}}_{a,b}(t)\frac{dadb}{a}\right\}\overline{g(t)}dt\\
&=\left\langle\int_{\mathbb{R}^+\times\mathbb{R}}\left(S^{M_{1},M_{2}}_\psi f\right)(a,b)\psi_{a,b}^M(t)\frac{dadb}{a},g(t)\right\rangle_{L^2(\mathbb{R})}.
\end{align*}
Since $g\in L^2(\mathbb{R})$ is arbitrary, we have
$$f(t)=\frac{1}{2\pi|B_{1}|C_{\psi,M_{1},M_{2}}}\int_{\mathbb{R}^+\times\mathbb{R}}\left(S^{M_{1},M_{2}}_\psi f\right)(a,b)\psi_{a,b}^M(t)\frac{dadb}{a}~\mbox{a.e.}$$
\end{proof}
The following theorem characterizes the range of the LCST and proves that the range is a reproducing kernel Hilbert space (RKHS). It also gives the explicit expression for the reproducing kernel.

\begin{theorem}\label{P4ReproducingKernelTheoremLCST}
Let $\psi$ be an admissible window and $f\in L^2\left(\mathbb{R}^+\times\mathbb{R},\frac{dadb}{a}\right)$ then $f\in\mathcal{S}^{M_{1},M_{2}}_{\psi}\left(L^2(\mathbb{R})\right)$ iff
\begin{equation}\label{P4ReproducingFormulaLCST}
f(c,d)=\frac{1}{2\pi|B_{1}|C_{\psi,M_{1},M_{2}}}\int_{\mathbb{R}^+\times\mathbb{R}}f(a,b)\left\langle\psi^{M_{1},M_{2}}_{a,b},\psi^{M_{1},M_{2}}_{c,d}\right\rangle_{L^2(\mathbb{R})}\frac{dadb}{a}.
\end{equation}
\end{theorem}
\begin{proof}
Let $f\in L^2\left(\mathbb{R}^+\times\mathbb{R},\frac{dadb}{a}\right),$ then there exists a function $h\in L^2(\mathbb{R})$ such that $S^{M_{1},M_{2}}_\psi h=f.$
\begin{align*}
f(c,d)&= \left(S^{M_{1},M_{2}}_\psi h\right)(c,d)\\
&=\int_{\mathbb{R}}h(t)\overline{\psi^{M_{1},M_{2}}_{c,d}(t)}dt\\
&=\frac{1}{2\pi|B_{1}|C_{\psi,M_{1},M_{2}}}\int_{\mathbb{R}^+\times\mathbb{R}}\left(S^{M_{1},M_{2}}_{\psi}h\right)(a,b)\left\{\int_{\mathbb{R}}\psi^{M_{1},M_{2}}_{a,b}(t)\overline{\psi^{M_{1},M_{2}}_{c,d}(t)}dt\right\}\frac{dadb}{a}\\
&=\frac{1}{2\pi|B_{1}|C_{\psi,M_{1},M_{2}}}\int_{\mathbb{R}^+\times\mathbb{R}}f(a,b)\left\langle\psi^{M_{1},M_{2}}_{a,b},\psi^{M_{1},M_{2}}_{c,d}\right\rangle_{L^2(\mathbb{R})}\frac{dadb}{a}.
\end{align*}
Conversely, let $f\in L^2\left(\mathbb{R}^+\times\mathbb{R},\frac{dadb}{a}\right),$ is such that
$$f(c,d)=\frac{1}{2\pi|B_{1}|C_{\psi,M_{1},M_{2}}}\int_{\mathbb{R}^+\times\mathbb{R}}f(a,b)\left\langle\psi^{M_{1},M_{2}}_{a,b},\psi^{M_{1},M_{2}}_{c,d}\right\rangle_{L^2(\mathbb{R})}\frac{dadb}{a}.$$
We claim that required $h$ is given by 
$$h(t)=\frac{1}{2\pi|B_{1}|C_{\psi,M_{1},M_{2}}}\int_{\mathbb{R}^+\times\mathbb{R}}f(a,b)\psi^{M_{1},M_{2}}_{a,b}(t)\frac{dadb}{a}.$$
To show that $h\in L^2(\mathbb{R})$:
\begin{align*}
\|h\|^2_{L^2(\mathbb{R})}&=\int_{\mathbb{R}}h(t)\overline{h(t)}dt\\
&=\frac{1}{\left(2\pi|B_{1}|C_{\psi,M_{1},M_{2}}\right)^2}\int_{\mathbb{R}^+\times\mathbb{R}}\left\{\int_{\mathbb{R}^+\times\mathbb{R}}f(a,b)\langle \psi^{M_{1},M_{2}}_{a,b},\psi^{M_{1},M_{2}}_{a',b'}\rangle_{L^2(\mathbb{R})}\frac{dadb}{a}\right\}\overline{f(a',b')}\frac{da'db'}{a'}\\
&=\frac{1}{\left(2\pi|B_{1}|C_{\psi,M_{1},M_{2}}\right)^2}\int_{\mathbb{R}^+\times\mathbb{R}}\left|f(a',b')\right|^2\frac{da'db'}{a'}\\
&<\infty,~~~\mbox{since}~~f\in L^2\left(\mathbb{R}^+\times\mathbb{R},\frac{dadb}{a}\right).
\end{align*}
Therefore, $h\in L^2(\mathbb{R}).$
Now, 
\begin{align*}
\left(S^{M_{1},M_{2}}_{\psi}h\right)(a,b)&=\int_{\mathbb{R}}h(t)\overline{\psi^{M_{1},M_{2}}_{a,b}(t)}dt\\
&=\int_{\mathbb{R}}\left\{\frac{1}{2\pi|B_{1}|C_{\psi,M_{1},M_{2}}}\int_{\mathbb{R}^+\times\mathbb{R}}f(a',b')\psi^{M_{1},M_{2}}_{a',b'}(t)\frac{da'db'}{a'}\right\}\overline{\psi^{M_{1},M_{2}}_{a,b}(t)}dt\\
&=\frac{1}{2\pi|B_{1}|C_{\psi,M_{1},M_{2}}}\int_{\mathbb{R}^+\times\mathbb{R}}f(a',b')\left\langle\psi^{M_{1},M_{2}}_{a',b'},\psi^{M_{1},M_{2}}_{a,b}\right\rangle_{L^2(\mathbb{R})}\frac{da'db'}{a'}\\
&=f(a,b).
\end{align*}
This finishes the proof.
\end{proof}
\begin{corollary}
The subspace  $S^{M_{1},M_{2}}_{\psi}\left(L^2(\mathbb{R})\right)$ of  $L^2\left(\mathbb{R}^+\times\mathbb{R},\frac{dadb}{a}\right),$ for an admissable window $\psi,$ is a RKHS with the kernel
$$K^{M_{1},M_{2}}_\psi(a,b;c,d)=\frac{1}{2\pi|B_{1}|C_{\psi,M_{1},M_{2}}}\left\langle \psi^{M_{1},M_{2}}_{a,b},\psi^{M_{1},M_{2}}_{c,d}\right\rangle_{L^2(\mathbb{R})}.$$
 \end{corollary}
\section{Multiresolution Analysis associated with the LCST}
Multiresolution analysis (MRA) is an efficient tool for the construction of an orthonormal wavelet basis for $L^2(\mathbb{R}).$ This ground breaking idea was initiated first in \cite{mallat1989multiresolution}. Recently, several authors have developed several MRAs and the corresponding orthonormal basis (ONB) (see \cite{dai2017new},\cite{shah2022special}). In this section we develop a novel MRA and discuss the construction of ONB for $L^2(\mathbb{R}).$ 

\begin{definition}\label{P4DefinitionMRA}
A sequence of closed subspace $V^{M_{1},M_{2}}_{m}\subset  L^2(\mathbb{R}),~m\in\mathbb{Z}$ is called an orthogonal MRA associated with the LCST if it satisfy the following
\begin{enumerate}
\item\label{P4DefnMRA01} $V^{M_{1},M_{2}}_{m}\subset V^{M_{1},M_{2}}_{m+1},~\forall m\in\mathbb{Z}$
\item\label{P4DefnMRA02} $f(t)\in V^{M_{1},M_{2}}_{m}\iff f(2t)e^{\frac{iA_{1}}{2B_{1}}\left\{(2t)^2-t^2\right\}}\in V^{M_{1},M_{2}}_{m+1},~\forall m\in\mathbb{Z}$
\item\label{P4DefnMRA03} $\displaystyle \bigcap_{m\in\mathbb{Z}}V^{M_{1},M_{2}}_{m}=\{0\},~\overline{\bigcup_{m\in\mathbb{Z}}V^{M_{1},M_{2}}_{m}}=L^2(\mathbb{R})$
\item\label{P4DefnMRA04} there exists $\phi$ in $ L^2(\mathbb{R}),$ called the scaling function, for which $\{\phi_{M_{1},M_{2},0,n}(t)\}_{n\in\mathbb{Z}}$  form an orthonormal basis (ONB) of $V^{M_{1},M_{2}}_{0},$ where 
\begin{align}\label{P4MRA040}
\phi_{M_{1},M_{2},0,n}=\phi(t-n)e^{-\frac{i}{2}\left\{\left(t^2-n^2\right)\frac{A_{1}}{B_{1}}-(t-n)^2\frac{A_{2}}{B_{2}}-\frac{2(t+n)}{B_{1}}\right\}}.
\end{align}
\end{enumerate} 
\end{definition}
\begin{theorem}
Let $\left\{V^{M_{1},M_{2}}_{m}\right\}_{m\in\mathbb{Z}}$ be an orthogonal MRA of $L^2(\mathbb{R})$ with $\phi$ as an scaling function. Then $\left\{\phi_{M_{1},M_{2},m,n}\right\}_{n\in\mathbb{Z}}$ form an ONB of $V^{M_{1},M_{2}}_{m},~m\in\mathbb{Z},$  where
\begin{align}\label{P4TranslatedandDilatesofphi}
\phi_{M_{1},M_{2},m,n}(t)=2^{\frac{m}{2}}\phi(2^mt-n)e^{-\frac{i}{2}\left\{\left(t^2-\left(\frac{n}{2^m}\right)^2\right)\frac{A_{1}}{B_{1}}-\left(2^mt-n\right)^2\frac{A_{2}}{B_{2}}-\frac{2(2^mt+n)}{B_{1}}\right\}},~n\in\mathbb{Z}.
\end{align} 
\end{theorem}
\begin{proof}
Using \eqref{P4TranslatedandDilatesofphi}, we have
\begin{align*}
\langle \phi_{M_{1},M_{2},m,n},&\phi_{M_{1},M_{2},m,k}\rangle_{L^2(\mathbb{R})}\\
&=2^me^{\frac{i}{B_1}(n-k)}\int_{\mathbb{R}}\phi(2^mt-n)\overline{\phi(2^mt-k)}e^{-\frac{i}{2}\left\{\left(t^2-\left(\frac{n}{2^m}\right)^2\right)\frac{A_{1}}{B_{1}}-\left(2^mt-n\right)^2\frac{A_{2}}{B_{2}}\right\}+\frac{i}{2}\left\{\left(t^2-\left(\frac{k}{2^m}\right)^2\right)\frac{A_{1}}{B_{1}}-\left(2^mt-k\right)^2\frac{A_{2}}{B_{2}}\right\}}dt\\
&=e^{\frac{i}{B_1}(n-k)-\frac{i}{2}\left\{\left(\frac{k}{2^m}\right)^2-\left(\frac{n}{2^m}\right)^2\right\}\frac{A_{1}}{B_{1}}}\int_{\mathbb{R}}\phi(t-n)\overline{\phi(t-k)}e^{-\frac{i}{2}\left\{-\left(t-n\right)^2\frac{A_{2}}{B_{2}}+\left(t-k\right)^2\frac{A_{2}}{B_{2}}\right\}}dt.
\end{align*}
Since $\left\{\phi_{M_{1},M_{2},0,n}\right\}_{n\in\mathbb{Z}}$ forms an ONB of $V^{M_{1},M_{2}}_{0},$ so
\begin{align*}
\int_{\mathbb{R}}\phi_{M_{1},M_{2},0,n}(t)\overline{\phi_{M_{1},M_{2},0,k}(t)}dt=\delta(n-k).
\end{align*}
This implies
$$e^{\frac{i}{B_1}(n-k)}\int_{\mathbb{R}}\phi(t-n)\overline{\phi(t-k)}e^{-\frac{i}{2}\left\{\left(k^2-n^2\right)\frac{A_{1}}{B_{1}}-\left(t-n\right)^2\frac{A_{2}}{B_{2}}+(t-k)^2\frac{A_{2}}{B_{2}}\right\}}=\delta(n-k).$$
This further gives
$$e^{\frac{i}{B_1}(n-k)-\frac{i}{2}\left(k^2-n^2\right)\frac{A_{1}}{B_{1}}}\int_{\mathbb{R}}\phi(t-n)\overline{\phi(t-k)}e^{-\frac{i}{2}\left\{-\left(t-n\right)^2\frac{A_{1}}{B_{1}}+(t-k)^2\frac{A_{2}}{B_{2}}\right\}}dt=
\begin{cases}
0,~n\neq k\\
1,~n=k.
\end{cases}
$$
$$\mbox{i.e.,} \int_{\mathbb{R}}\phi(t-n)\overline{\phi(t-n)}dt=1$$
$$\mbox{and}\int_{\mathbb{R}}\phi(t-n)\overline{\phi(t-k)}e^{-\frac{i}{2}\{-(t-n)^2\frac{A_{1}}{B_{1}}+(t-k)^2\frac{A_{2}}{B_{2}}\}}dt=0,~if~n\neq k.$$
Now 
\begin{align*}
\Big\langle \phi_{M_{1},M_{2},m,n},&\phi_{M_{1},M_{2},m,k}\Big\rangle_{L^2(\mathbb{R})}\\
&=
\begin{cases}
\displaystyle e^{\frac{i}{B_1}(n-k)-\frac{i}{2}\left\{\left(\frac{k}{2^m}\right)^2-\left(\frac{n}{2^m}\right)^2\right\}\frac{A_{1}}{B_{1}}}\int_{\mathbb{R}}\phi(t-n)\overline{\phi(t-k)}e^{-\frac{i}{2}\left\{-(t-n)^2\frac{A_{2}}{B_{2}}+(t-k)^2\frac{A_{2}}{B_{2}}\right\}}dt,~if~n\neq k\\
\displaystyle\int_{\mathbb{R}}\phi(t-n)\overline{\phi(t-k)}dt,~if~n=k
\end{cases}\\
&=
\begin{cases}
0,~if~n\neq k\\
1,~if ~n=k.
\end{cases}
\end{align*}
Therefore, $\left\{\phi_{M_{1},M_{2},m,n}\right\}_{n\in\mathbb{Z}}$ form an orthonormal system.
Let $f\in V^{M_{1},M_{2}}_{m},$ then according to (\ref{P4DefnMRA02}) of the definition \ref{P4DefinitionMRA}, we have 
$$f(2^{-m}t)e^{\frac{i}{2}\left\{\left(2^{-m}t\right)^2-t^2\right\}\frac{A_{1}}{B_{1}}}\in V^{M_{1},M_{2}}_{0}.$$
Since $\left\{\phi_{M_{1},M_{2},0,n}\right\}_{n\in\mathbb{Z}}$ is an ONB of $V^{M_{1},M_{2}}_{0},$ so we have
$$f\left(2^{-m}t\right)e^{\frac{i}{2}\left\{\left(2^{-m}t\right)^2-t^2\right\}\frac{A_{1}}{B_{1}}}=\sum_{n\in\mathbb{Z}}c_{n}\phi(t-n)e^{-\frac{i}{2}\left\{\left(t^2-n^2\right)\frac{A_{1}}{B_{1}}-(t-n)^2\frac{A_{2}}{B_{2}}-\frac{2(t+n)}{B_{1}}\right\}}.$$
This implies 
$$f\left(2^{-m}t\right)=\sum_{n\in\mathbb{Z}}c_{n}\phi(t-n)e^{-\frac{i}{2}\left\{\left\{\left(2^{-m}t\right)^2-n^2\right\}\frac{A_{1}}{B_{1}}-(t-n)^2\frac{A_{2}}{B_{2}}-\frac{2(t+n)}{B_{1}}\right\}}.$$
Now,
\begin{align*}
f(t)
&=\sum_{n\in\mathbb{Z}}c_{n}\phi(2^mt-n)e^{-\frac{i}{2}\left\{\left(t^2-n^2\right)\frac{A_{1}}{B_{1}}-\left(2^mt-n\right)^2\frac{A_{2}}{B_{2}}-\frac{2(2^mt+n)}{B_{1}}\right\}}\\
&=\sum_{n\in\mathbb{Z}}c'_{n}\phi_{M_{1},M_{2},m,n}(t),~\mbox{where}~c'_{n}=2^{-\frac{m}{2}}c_ne^{-\frac{i}{2}\left\{\left(\frac{n}{2^m}\right)^2-n^2\right\}}.
\end{align*}
Therefore, 
$$V^{M_{1},M_{2}}_{m}=\overline{\mbox{span}}\left\{\phi_{M_{1},M_{2},m,n}\right\}_{n\in\mathbb{Z}},$$
i.e., $\left\{\phi_{M_{1},M_{2},m,n}\right\}_{n\in\mathbb{Z}}$ forms an ONB of $V^{M_{1},M_{2}}_{m}.$
\end{proof}
If the orthogonality of $\left\{\phi_{M_{1},M_{2},0,n}\right\}_{n\in\mathbb{Z}}$ in condition (\ref{P4DefnMRA04}) of Definition \ref{P4DefinitionMRA} is replaced by assumption that it is a Riesz basis of $V_{0}^{M_{1},M_{2}},$ then $\phi(t)$ generates a so called generalized MRA of $L^2(\mathbb{R})$ associated with the LCST.
\begin{theorem}
Let $\phi\in L^2(\mathbb{R})$ and $V^{M_{1},M_{2}}_{0}=\overline{\mbox{span}}\left\{\phi_{M_{1},M_{2},0,n}(t)\right\}_{n\in\mathbb{Z}},$ where $\phi_{M_{1},M_{2},0,n}(t)$ is defined as \eqref{P4MRA040}, then $\left\{\phi_{M_{1},M_{2},0,n}\right\}_{n\in\mathbb{Z}}$ forms a Riesz basis of $V^{M_{1},M_{2}}_{0}$ iff $\exists$ $A,B\in\mathbb{R}^+$ with $A\leq B$ such that 
$$ A\leq \sum_{k\in\mathbb{Z}}\left|\left(\mathcal{L}^{M_{2}}\left\{e^{\frac{it}{B_{1}}}\phi(t)\right\}\right)\left(\frac{B_{2}}{B_{1}}\left(u+2k\pi|B_{1}|\right)\right)\right|^2\leq B,~\forall u\in[0,2\pi|B_1|].$$ 
\end{theorem}
\begin{proof}
For any $f\in V^{M_{1},M_{2}}_{0},$ we have 
$$f(t)=\sum_{n\in\mathbb{Z}}c_{n}\phi_{M_{1},M_{2},0,n}(t).$$
This gives
\begin{align*}
\left(\mathcal{L}^{M_{1}}f\right)(u)=\sum_{n\in\mathbb{Z}}c_{n}\left(\mathcal{L}^{M_{1}}\phi_{M_{1},M_{2},0,n}\right)(u).
\end{align*}
Since 
$$\left(\mathcal{L}^{M_{1}}\phi_{M_{1},M_{2},0,n}\right)(u)=\sqrt{2\pi iB_{2}}e^{\frac{2in}{B_{1}}-\frac{iD_{2}}{2B_{2}}\left(\frac{B_{2}u}{B_{1}}\right)^2}K_{M_{1}}(n,u)\left(\mathcal{L}^{M_{2}}\left\{e^{\frac{it}{B_{1}}}\phi(t)\right\}\right)\left(\frac{B_{2}u}{B_{1}}\right),$$ 
we have
\begin{align*}
\left(\mathcal{L}^{M_{1}}f\right)(u)
&=\sqrt{2\pi iB_{2}}e^{-\frac{iD_{2}}{2B_{2}}\left(\frac{B_{2}u}{B_{1}}\right)^2}\sum_{n\in\mathbb{Z}}c_{n}e^{\frac{2in}{B_{1}}}K_{M_{1}}(n,u)\left(\mathcal{L}^{M_{2}}\left\{e^{\frac{it}{B_{1}}}\phi(t)\right\}\right)\left(\frac{B_{2}u}{B_{1}}\right)\\
&=\sqrt{2\pi iB_{2}}e^{-\frac{iD_{2}}{2B_{2}}\left(\frac{B_{2}u}{B_{1}}\right)^2}\tilde{\mathcal{L}}^{M_{1}}\left[c_{n}e^{\frac{2in}{B_{1}}}\right](u)\left(\mathcal{L}^{M_{2}}\left\{e^{\frac{it}{B_{1}}}\phi(t)\right\}\right)\left(\frac{B_{2}u}{B_{1}}\right)
\end{align*}
$$\mbox{i.e.,}~\left(\mathcal{L}^{M_{1}}f\right)(u)=\sqrt{2\pi iB_{2}}e^{-\frac{iD_{2}}{2B_{2}}\left(\frac{B_{2}u}{B_{1}}\right)^2}\tilde{\mathcal{L}}^{M_{1}}\left[c_{n}e^{\frac{2in}{B_{1}}}\right](u)\left(\mathcal{L}^{M_{2}}\left\{e^{\frac{it}{B_{1}}}\phi(t)\right\}\right)\left(\frac{B_{2}u}{B_{1}}\right),$$
where $\tilde{\mathcal{L}}^{M_{1}}\left[c_{n}e^{\frac{2in}{B_{1}}}\right](u)=\displaystyle\sum_{n\in\mathbb{Z}}c_{n}e^{\frac{2in}{B_{1}}}K_{M_{1}}(n,u).$\\
Now, 
\begin{align*}
\|f\|^2_{L^2(\mathbb{R})}&=\left\|\mathcal{L}^{M_{1}} f\right\|^2_{L^2(\mathbb{R})}=2\pi|B_{2}|\int_{\mathbb{R}}\left|\tilde{\mathcal{L}}^{M_{1}}\left[c_{n}e^{\frac{2in}{B_{1}}}\right](u)\right|^2\left|\left(\mathcal{L}^{M_{2}}\left\{e^{\frac{it}{B_{1}}}\phi(t)\right\}\right)\left(\frac{B_{2}u}{B_{1}}\right)\right|^2du\\
&=2\pi|B_{2}|\sum_{k\in\mathbb{Z}}\int_{2k\pi|B_{1}|}^{2(k+1)\pi|B_{1}|}\left|\tilde{\mathcal{L}}^{M_{1}}\left[c_{n}e^{\frac{2in}{B_{1}}}\right](u)\right|^2\left|\left(\mathcal{L}^{M_{2}}\left\{e^{\frac{it}{B_{1}}}\phi(t)\right\}\right)\left(\frac{B_{2}u}{B_{1}}\right)\right|^2du\\
&=2\pi|B_{2}|\sum_{k\in\mathbb{Z}}\int_{0}^{2\pi|B_{1}|}\left|\tilde{\mathcal{L}}^{M_{1}}\left[c_{n}e^{\frac{2in}{B_{1}}}\right]\left(u+2k\pi|B_{1}|\right)\right|^2\left|\left(\mathcal{L}^{M_{2}}\left\{e^{\frac{it}{B_{1}}}\phi(t)\right\}\right)\left(\frac{B_{2}\left(u+2k\pi|B_{1}|\right)}{B_{1}}\right)\right|^2du\\
&=2\pi|B_{2}|\sum_{k\in\mathbb{Z}}\int_{0}^{2\pi|B_{1}|}\left|\tilde{\mathcal{L}}^{M_{1}}\left[c_{n}e^{\frac{2in}{B_{1}}}\right](u)\right|^2\left|\left(\mathcal{L}^{M_{2}}\left\{e^{\frac{it}{B_{1}}}\phi(t)\right\}\right)\left(\frac{B_{2}\left(u+2k\pi|B_{1}|\right)}{B_{1}}\right)\right|^2du,\\
&\hspace{7cm}~\mbox{since}~ \left|\tilde{\mathcal{L}}^{M_{1}}\left[c_{n}e^{\frac{2in}{B_{1}}}\right](u)\right|~ \mbox{is}~ 2k\pi|B_{1}|~\mbox{periodic}.
\end{align*}
i.e., 
\begin{align}\label{P4theo.Eqn1}
\|f\|^2_{L^2(\mathbb{R})}=2\pi|B_{2}|\int_{0}^{2\pi|B_{1}|}\left|\tilde{\mathcal{L}}^{M_{1}}\left[c_{n}e^{\frac{2in}{B_{1}}}\right](u)\right|^2\sum_{k\in\mathbb{Z}}\left|\left(\mathcal{L}^{M_{2}}\left\{e^{\frac{it}{B_{1}}}\phi(t)\right\}\right)\left(\frac{B_{2}(u+2k\pi|B_{1}|)}{B_{1}}\right)\right|^2du.
\end{align}
It is given that 
\begin{align}\label{P4theo.Eqn2}
A\leq \sum_{k\in\mathbb{Z}}\left|\left(\mathcal{L}^{M_{2}}\left\{e^{\frac{it}{B_{1}}}\phi(t)\right\}\right)\left(\frac{B_{2}}{B_{1}}\left(u+2k\pi|B_{1}|\right)\right)\right|^2\leq B
\end{align}
and also we have 
$$\|c'_{n}\|^2_{l^2(\mathbb{Z})}=\int_{0}^{2\pi|B_{1}|}\left|\tilde{\mathcal{L}}^{M_{1}}\left[c_{n}e^{\frac{2in}{B_{1}}}\right](u)\right|^2du,~\mbox{where}~c'_{n}=c_{n}e^{\frac{2in}{B_{1}}}$$ 
i.e., 
\begin{align}\label{P4theo.Eqn3}
\|c_{n}\|^2_{l^2(\mathbb{Z})}=\int_{0}^{2\pi|B_{1}|}\left|\tilde{\mathcal{L}}^{M_{1}}\left[c_{n}e^{\frac{2in}{B_{1}}}\right](u)\right|^2du.
\end{align}
Therefore, we have 
$$A \left|\tilde{\mathcal{L}}^{M_{1}}\left[c_{n}e^{\frac{2in}{B_{1}}}\right](u)\right|^2\leq\left|\tilde{\mathcal{L}}^{M_{1}}\left[c_{n}e^{\frac{2in}{B_{1}}}\right](u)\right|^2\sum_{k\in\mathbb{Z}}\left|\left(\mathcal{L}^{M_{2}}\left\{e^{\frac{it}{B_{1}}}\phi(t)\right\}\right)\left(\frac{B_{2}\left(u+2k\pi|B_{1}|\right)}{B_{1}}\right)\right|^2\leq B\left|\tilde{\mathcal{L}}^{M_{1}}\left[c_{n}e^{\frac{2in}{B_{1}}}\right](u)\right|^2.$$
This implies

\begin{align*}
A \int_{0}^{2\pi|B_{1}|}\left|\tilde{\mathcal{L}}^{M_{1}}\left[c_{n}e^{\frac{2in}{B_{1}}}\right](u)\right|^2du\leq\int_{0}^{2\pi|B_1|}\left|\tilde{\mathcal{L}}^{M_{1}}\left[c_{n}e^{\frac{2in}{B_{1}}}\right](u)\right|^2\sum_{k\in\mathbb{Z}}\Bigg|&\left(\mathcal{L}^{M_{2}}\left\{e^{\frac{it}{B_{1}}}\phi(t)\right\}\right)\left(\frac{B_{2}\left(u+2k\pi|B_{1}|\right)}{B_{1}}\right)\Bigg|^2du\\
&\leq B\int_{0}^{2\pi|B_{1}|} \left|\tilde{\mathcal{L}}^{M_{1}}\left[c_{n}e^{\frac{2in}{B_{1}}}\right](u)\right|^2du,
\end{align*}
i.e.,
$$2\pi|B_{2}| A \|c_{n}\|^2_{l^2(\mathbb{Z})}\leq\|f\|^2_{L^2(\mathbb{R})}\leq 2\pi|B_{2}|B\|c_{n}\|^2_{l^2(\mathbb{Z})}.$$
Thus, $\left\{\phi_{M_{1},M_{2},0,n}\right\}_{n\in\mathbb{Z}}$ forms a Riesz basis of $V^{M_{1},M_{2}}_{0}.$

Conversely, let $\left\{\phi_{M_{1},M_{2},0,n}\right\}_{n\in\mathbb{Z}}$ is a Riesz basis of $V^{M_{1},M_{2}}_{0}$ then
 $$A_{0} \|c_{n}\|^2_{l^2(\mathbb{Z})}\leq\left\|f(t)=\sum_{n\in\mathbb{Z}}c_{n}\phi_{M_{1},M_{2},0,n}(t)\right\|^2_{L^2(\mathbb{R})}\leq B_{0}\|c_{n}\|^2_{l^2(\mathbb{Z})}.$$
Since $f(t)=\sum_{n\in\mathbb{Z}}c_{n}\phi_{M_{1},M_{2},0,n}(t),$ so we have
 $$\|f\|^2_{L^2(\mathbb{R})}=2\pi|B_{2}|\int_{0}^{2\pi|B_{1}|}\left|\tilde{\mathcal{L}}^{M_{1}}\left[c_{n}e^{\frac{2in}{B_{1}}}\right](u)\right|^2\sum_{k\in\mathbb{Z}}\left|\left(\mathcal{L}^{M_{2}}\left\{e^{\frac{it}{B_{1}}}\phi(t)\right\}\right)\left(\frac{B_{2}\left(u+2k\pi|B_{1}|\right)}{B_{1}}\right)\right|^2du.$$
 This gives
 $$A_{0}\|c_{n}\|^2_{l^2(\mathbb{Z})}\leq 2\pi|B_{2}|\int_{0}^{2\pi|B_{1}|}\left|\tilde{\mathcal{L}}^{M_{1}}\left[c_{n}e^{\frac{2in}{B_{1}}}\right](u)\right|^2\sum_{k\in\mathbb{Z}}\left|\left(\mathcal{L}^{M_{2}}\left\{e^{\frac{it}{B_{1}}}\phi(t)\right\}\right)\left(\frac{B_{2}\left(u+2k\pi|B_{1}|\right)}{B_{1}}\right)\right|^2du\leq B_{0}\|c_{n}\|^2_{l^2(\mathbb{Z})},$$
which further implies that 
$$A_{0}\|c'_{n}\|^2_{l^2(\mathbb{Z})}\leq 2\pi|B_{2}|\int_{0}^{2\pi|B_{1}|}\left|\tilde{\mathcal{L}}^{M_{1}}\left[c_{n}e^{\frac{2in}{B_{1}}}\right](u)\right|^2\sum_{k\in\mathbb{Z}}\left|\left(\mathcal{L}^{M_{2}}\left\{e^{\frac{it}{B_{1}}}\phi(t)\right\}\right)\left(\frac{B_{2}\left(u+2k\pi|B_{1}|\right)}{B_{1}}\right)\right|^2du\leq B_{0}\|c'_{n}\|^2_{l^2(\mathbb{Z})},$$
where $c'_{n}=c_{n}e^{\frac{2in}{B_{1}}}.$
\begin{align*}
\frac{A_{0}}{2\pi|B_{2}|}\int_{0}^{2\pi|B_{1}|}\left|\tilde{\mathcal{L}}^{M_{1}}[c_{n}e^{\frac{2in}{B_{1}}}](u)\right|^2du\leq \int_{0}^{2\pi|B_{1}|}\left|\tilde{\mathcal{L}}^{M_{1}}\left[c_{n}e^{\frac{2in}{B_{1}}}\right](u)\right|^2\sum_{k\in\mathbb{Z}}&\left|\left(\mathcal{L}^{M_{2}}\left\{e^{\frac{it}{B_{1}}}\phi(t)\right\}\right)\left(\frac{B_{2}\left(u+2k\pi|B_{1}|\right)}{B_{1}}\right)\right|^2du\\
&\leq \frac{B_{0}}{2\pi|B_{2}|}\int_{0}^{2\pi|B_{1}|}\left|\tilde{\mathcal{L}}^{M_{1}}\left[c_{n}e^{\frac{2in}{B_{1}}}\right](u)\right|^2du.
\end{align*}
Because of the arbitrariness of $\tilde{\mathcal{L}}^{M_{1}}[c_{n}e^{\frac{2in}{B_{1}}}](u),$ we have 
$$\frac{A_{0}}{2\pi|B_{2}|}\leq \sum_{k\in\mathbb{Z}}\left|\left(\mathcal{L}^{M_{2}}\left\{e^{\frac{it}{B_{1}}}\phi(t)\right\}\right)\left(\frac{B_{2}\left(u+2k\pi|B_{1}|\right)}{B_{1}}\right)\right|^2\leq \frac{B_{0}}{2\pi|B_{2}|},~\mbox{a.e.}~u\in[0,2\pi|B_1|]$$
i.e., 
$$A\leq \sum_{k\in\mathbb{Z}}\left|\left(\mathcal{L}^{M_{2}}\left\{e^{\frac{it}{B_{1}}}\phi(t)\right\}\right)\left(\frac{B_{2}\left(u+2k\pi|B_{1}|\right)}{B_{1}}\right)\right|^2\leq B,~\mbox{a.e.}~u\in[0,2\pi|B_1|]$$
where $A=\frac{A_{0}}{2\pi|B_{2}|}$ and $B=\frac{B_{0}}{2\pi|B_{2}|}.$\\
In particular $\left\{\phi_{M_{1},M_{2},0,n}\right\}_{n\in\mathbb{Z}}$ is an ONB of $V^{M_{1},M_{2}}_{0}$ iff $A=B=\frac{1}{2\pi |B_{2}|}.$ 
\end{proof}
\begin{theorem}
Let $\left\{V^{M_{1},M_{2}}_{k}\right\}_{k\in\mathbb{Z}}$ be a generalized MRA, with the scaling function $\phi.$ Let 
\begin{align}\label{P4GeneralizedMRAtoOrthoMRA}
\left(\mathcal{L}^{M_{2}}\left\{e^{\frac{it}{B_{1}}}\psi(t)\right\}\right)\left(\frac{B_{2}u}{B_{1}}\right)=\frac{\left(\mathcal{L}^{M_{2}}\left\{e^{\frac{it}{B_{1}}}\phi(t)\right\}\right)\left(\frac{B_{2}u}{B_{1}}\right)}{\sqrt{2\pi|B_{2}|\displaystyle\sum_{k\in\mathbb{Z}}\left|\left(\mathcal{L}^{M_{2}}\left\{e^{\frac{it}{B_{1}}}\phi(t)\right\}\right)\left(\frac{B_{2}\left(u+2k\pi|B_{1}|\right)}{B_{1}}\right)\right|^2}}.
\end{align} 
Then the set $\left\{\psi_{M_{1},M_{2},0,n}(t)=\psi(t-n)e^{-\frac{i}{2}\left\{\left(t^2-n^2\right)\frac{A_{1}}{B_{1}}-(t-n)^2\frac{A_{2}}{B_{2}}-\frac{2(t+n)}{B_{1}}\right\}}\right\}_{n\in\mathbb{Z}}$ form an ONB of $V^{M_{1},M_{2}}_{0}.$
\end{theorem}
\begin{proof}
Using \eqref{P4GeneralizedMRAtoOrthoMRA} and $2\pi|B_1|$ periodicity of $\displaystyle\sum_{k\in\mathbb{Z}}\left|\left(\mathcal{L}^{M_{2}}\left\{e^{\frac{it}{B_{1}}}\phi(t)\right\}\right)\left(\frac{B_{2}\left(u+2k\pi|B_{1}|\right)}{B_{1}}\right)\right|^2$ we get
\begin{align*}
\psi_{M_{1},M_{2},0,0}(t)=\psi(t)e^{-\frac{i}{2}\left(t^2\frac{A_{1}}{B_{1}}-t^2\frac{A_{2}}{B_{2}}-\frac{2t}{B_{1}}\right)}\in V^{M_{1},M_{2}}_{0}.
\end{align*} 
So 
$$\psi_{M_{1},M_{2},0,0}(t)=\sum_{n\in\mathbb{Z}}c_{n}\phi_{M_{1},M_{2},0,n}(t)$$
i.e., 
$$\psi(t)e^{-\frac{i}{2}\left(t^2\frac{A_{1}}{B_{1}}-t^2\frac{A_{2}}{B_{2}}-\frac{2t}{B_{1}}\right)}=\sum_{n\in\mathbb{Z}}c_{n}\phi(t-n)e^{-\frac{i}{2}\left\{\left(t^2-n^2\right)\frac{A_{1}}{B_{1}}-(t-n)^2\frac{A_{2}}{B_{2}}-\frac{2(t+n)}{B_{1}}\right\}}.$$
This gives
\begin{align*}
\psi(t)e^{\frac{it}{B_{1}}}
&=\sum_{n\in\mathbb{Z}}c_{n}\phi(t-n)e^{-\frac{i}{2}\left\{\left(t^2-n^2\right)\frac{A_{1}}{B_{1}}-\left(t-n\right)^2\frac{A_{2}}{B_{2}}-\frac{2(t+n)}{B_{1}}\right\}}e^{\frac{i}{2}\left(t^2\frac{A_{1}}{B_{1}}-t^2\frac{A_{2}}{B_{2}}\right)}\\
&=\sum_{n\in\mathbb{Z}}c_{n}\phi(t-n)e^{\frac{it}{B_{1}}}e^{-\frac{i}{2}\left(-n^2\frac{A_{1}}{B_{1}}-n^2\frac{A_{2}}{B_{2}}+2tn\frac{A_{2}}{B_{2}}-\frac{2n}{B_1}\right)}\\
&=\sum_{n\in\mathbb{Z}}c_{n}e^{\frac{it}{B_{1}}-itn\frac{A_{2}}{B_{2}}}\phi(t-n)e^{\frac{i}{2}\left\{\left(\frac{A_{1}}{B_{1}}+\frac{A_{2}}{B_{2}}\right)n^2+\frac{2n}{B_1}\right\}}
\end{align*}
This implies
\begin{align}\label{P4theo4.3.Eqn1}
\left(\mathcal{L}^{M_{2}}\left\{\psi(t)e^{\frac{it}{B_{1}}}\right\}\right)(u)=\sum_{n\in\mathbb{Z}}c_{n}e^{\frac{i}{2}\left\{\left(\frac{A_{1}}{B_{1}}+\frac{A_{2}}{B_{2}}\right)n^2+\frac{2n}{B_1}\right\}}\left(\mathcal{L}^{M_{2}}\left\{e^{\frac{it}{B_{1}}-itn\frac{A_{2}}{B_{2}}}\phi(t-n)\right\}\right)(u).
\end{align}
Now, 
\begin{align*}
\left(\mathcal{L}^{M_{2}}\left\{e^{\frac{it}{B_{1}}-itn\frac{A_{2}}{B_{2}}}\phi(t-n)\right\}\right)(u)
&=\int_{\mathbb{R}}e^{\frac{it}{B_{1}}-itn\frac{A_{2}}{B_{2}}}\phi(t-n)e^{\frac{i}{2}\left(\frac{A_{2}}{B_{2}}t^2-\frac{2tu}{B_{2}}+\frac{D_{2}}{B_{2}}t^2\right)}dt\\
&=\int_{\mathbb{R}}\phi(t)e^{\frac{it}{B_{1}}+\frac{in}{B_{1}}-in^2\frac{A_{2}}{B_{2}}+\frac{i}{2}\left(\frac{A_{2}}{B_{2}}n^2-\frac{2nu}{B_{2}}\right)}K_{M_{2}}(t,u)dt\\
&=e^{\frac{in}{B_{1}}-in^2\frac{A_{2}}{2B_{2}}-\frac{inu}{B_{2}}}\left(\mathcal{L}^{M_{2}}\left\{e^{\frac{it}{B_{1}}}\phi(t)\right\}\right)(u).
\end{align*}
Therefore, we have
\begin{align*}
\left(\mathcal{L}^{M_{2}}\left\{\psi(t)e^{\frac{it}{B_{1}}}\right\}\right)(u)
&=\sum_{n\in\mathbb{Z}}c_{n}e^{\frac{i}{2}\left\{\left(\frac{A_{1}}{B_{1}}+\frac{A_{2}}{B_{2}}\right)n^2+\frac{2n}{B_1}\right\}}e^{\frac{in}{B_{1}}-in^2\frac{A_{2}}{2B_{2}}-\frac{inu}{B_{2}}}\left(\mathcal{L}^{M_{2}}\left\{e^{\frac{it}{B_{1}}}\phi(t)\right\}\right)(u)\\
&=\sum_{n\in\mathbb{Z}}c_{n}e^{in^2\frac{A_{2}}{2B_{2}}+\frac{2in}{B_{1}}-\frac{inu}{B_{2}}}\left(\mathcal{L}^{M_{2}}\left\{e^{\frac{it}{B_{1}}}\phi(t)\right\}\right)(u)\\
&=\sum_{n\in\mathbb{Z}}c'_{n}e^{-\frac{inu}{B_{2}}}\left(\mathcal{L}^{M_{2}}\left\{e^{\frac{it}{B_{1}}}\phi(t)\right\}\right)(u),~\mbox{where}~c'_{n}=c_{n}e^{in^2\frac{A_{2}}{2B_{2}}+\frac{2in}{B_{1}}}\\
&=\left(\mathcal{L}^{M_{2}}\left\{e^{\frac{it}{B_{1}}}\phi(t)\right\}\right)(u)\sum_{n\in\mathbb{Z}}c'_{n}e^{-in\left(\frac{u}{B_{2}}\right)}\\
&=\left(\mathcal{L}^{M_{2}}\left\{e^{\frac{it}{B_{1}}}\phi(t)\right\}\right)(u)\tilde{\mathcal{F}}[c'_{n}]\left(\frac{u}{B_{2}}\right),~\mbox{where}~\tilde{\mathcal{F}}[c'_{n}]\left(\frac{u}{B_{2}}\right)=\sum_{n\in\mathbb{Z}}c'_{n}e^{-in\left(\frac{u}{B_{2}}\right)}.
\end{align*}
i.e.,
$$\left(\mathcal{L}^{M_{2}}\left\{\psi(t)e^{\frac{it}{B_{1}}}\right\}\right)\left(\frac{B_{2}u}{B_{1}}\right)=\left(\mathcal{L}^{M_{2}}\left\{e^{\frac{it}{B_{1}}}\phi(t)\right\}\right)\left(\frac{B_{2}u}{B_{1}}\right)\tilde{\mathcal{F}}[c'_{n}]\left(\frac{u}{B_{1}}\right).$$
This gives
$$\tilde{\mathcal{F}}[c'_{n}]\left(\frac{u}{B_{1}}\right)=\frac{\left(\mathcal{L}^{M_{2}}\left\{e^{\frac{it}{B_{1}}}\psi(t)\right\}\right)\left(\frac{B_{2}u}{B_{1}}\right)}{\left(\mathcal{L}^{M_{2}}\left\{e^{\frac{it}{B_{1}}}\phi(t)\right\}\right)\left(\frac{B_{2}u}{B_{1}}\right)},$$ 
so using the given hypothesis we have
\begin{align}\label{P4theo4.3.Eqn2}
\tilde{\mathcal{F}}[c'_{n}]\left(\frac{u}{B_{1}}\right)=\frac{1}{\sqrt{2\pi|B_{2}|\displaystyle\sum_{k\in\mathbb{Z}}\left|\left(\mathcal{L}^{M_{2}}\left\{e^{\frac{it}{B_{1}}}\phi(t)\right\}\right)\left(\frac{B_{2}\left(u+2k\pi|B_{1}|\right)}{B_{1}}\right)\right|^2}}.
\end{align}
Since $\tilde{\mathcal{F}}[c'_{n}]\left(\frac{u}{B_{1}}\right)$ is $2\pi|B_{1}|$ periodic, we have
$$\left|\left(\mathcal{L}^{M_{2}}\left\{\psi(t)e^{\frac{it}{B_{1}}}\right\}\right)\left(\frac{B_{2}\left(u+2k\pi|B_{1}|\right)}{B_{1}}\right)\right|^2=\left|\left(\mathcal{L}^{M_{2}}\left\{e^{\frac{it}{B_{1}}}\phi(t)\right\}\right)\left(\frac{B_{2}\left(u+2k\pi|B_{1}|\right)}{B_{1}}\right)\right|^2\left|\tilde{\mathcal{F}}[c'_{n}]\left(\frac{u}{B_{1}}\right)\right|^2$$
i.e., 
\begin{align}\label{P4theo4.3.Eqn3}
\sum_{n\in\mathbb{Z}}\left|\left(\mathcal{L}^{M_{2}}\left\{\psi(t)e^{\frac{it}{B_{1}}}\right\}\right)\left(\frac{B_{2}\left(u+2k\pi|B_{1}|\right)}{B_{1}}\right)\right|^2=\left|\tilde{\mathcal{F}}[c'_{n}]\left(\frac{u}{B_{1}}\right)\right|^2\sum_{n\in\mathbb{Z}}\left|\left(\mathcal{L}^{M_{2}}\left\{e^{\frac{it}{B_{1}}}\phi(t)\right\}\right)\left(\frac{B_{2}\left(u+2k\pi|B_{1}|\right)}{B_{1}}\right)\right|^2.
\end{align}
Thus, from \eqref{P4theo4.3.Eqn2} and \eqref{P4theo4.3.Eqn3}, we have 
$$\sum_{n\in\mathbb{Z}}\left|\left(\mathcal{L}^{M_{2}}\left\{\psi(t)e^{\frac{it}{B_{1}}}\right\}\right)\left(\frac{B_{2}(u+2k\pi|B_{1}|)}{B_{1}}\right)\right|^2=\frac{1}{2\pi|B_{2}|}.$$
This shows that $\left\{\psi_{M_{1},M_{2},0,n}\right\}_{n\in\mathbb{Z}}$ is an ONB of $V^{M_{1},M_{2}}_{0}.$
\end{proof}

\subsection{Construction of orthonormal linear canonical wavelets}
Based on the developed MRA, the orthonormal linear canonical wavelet basis can be constructed. Let $W_{k}^{M_{1},M_{2}}$ be the orthogonal complement of $V_{k}^{M_{1},M_{2}}$ in $V_{k+1}^{M_{1},M_{2}},$ i.e.,
$$W_{k}^{M_{1},M_{2}}\perp V_{k}^{M_{1},M_{2}},~V_{k+1}^{M_{1},M_{2}}=W_{k}^{M_{1},M_{2}}\bigoplus V_{k}^{M_{1},M_{2}},~\forall~k\in\mathbb{Z}.$$ \\
Then using Definition \ref{P4DefinitionMRA}, we have the following properties of $\left\{W_{k}^{M_{1},M_{2}}\right\}_{k\in\mathbb{Z}}.$
\begin{enumerate}
\item\label{P4DefinitionMRA1} $W_{k}^{M_{1},M_{2}}\perp V_{k}^{M_{1},M_{2}},~\forall~k\neq l;$
\item\label{P4DefinitionMRA2} $\displaystyle \bigoplus_{k\in\mathbb{Z}}W_{k}^{M_{1},M_{2}}=L^2(\mathbb{R});$
\item\label{P4DefinitionMRA3} $g(t)\in W_{k}^{M_{1},M_{2}}\iff g(2t)e^{\frac{iA_{1}}{2B_{1}}\left\{(2t)^2-t^2\right\}}\in W_{k+1}^{M_{1},M_{2}},~\forall~k\in\mathbb{Z}.$
\end{enumerate}
By determining an ONB of the subspace $W_{k}^{M_{1},M_{2}},$ one can construct an ONB of $L^2(\mathbb{R})$ according to property \eqref{P4DefinitionMRA2}.

According to property \eqref{P4DefinitionMRA3} these orthonormal bases are constructed by finding out the same for $W_{0}^{M_{1},M_{2}}.$ Thus, the crucial part is constructing $\psi(t)\in L^2(\mathbb{R})$ for which $\left\{\psi_{M_{1},M_{2},0,n}=\psi(t-n)e^{-\frac{i}{2}\left\{\left(t^2-n^2\right)\frac{A_{1}}{B_{1}}-\left(t-n\right)^2\frac{A_{2}}{B_{2}}-\frac{2(t+n)}{B_{1}}\right\}}\right\}_{n\in\mathbb{Z}}$ form an ONB of $W_{0}^{M_{1},M_{2}}.$ 
We call the basis $\{\psi_{M_1,M_2,m,n}(t)\}_{m,n\in\mathbb{Z}},$ where
\begin{align*}
\psi_{M_{1},M_{2},m,n}(t)=2^{\frac{m}{2}}\psi(2^mt-n)e^{-\frac{i}{2}\left\{\left(t^2-\left(\frac{n}{2^m}\right)^2\right)\frac{A_{1}}{B_{1}}-\left(2^mt-n\right)^2\frac{A_{2}}{B_{2}}-\frac{2(2^mt+n)}{B_{1}}\right\}},~m,n\in\mathbb{Z},
\end{align*} 
the orthonormal linear canonical wavelet basis of $L^2(\mathbb{R}).$

Since $\left\{\phi_{M_{1},M_{2},1,n}(t)=2^{\frac{1}{2}}\phi(2t-n)e^{-\frac{i}{2}\left\{\left(t^2-\left(\frac{n}{2}\right)^2\right)\frac{A_{1}}{B_{1}}-(2t-n)^2\frac{A_{2}}{B_{2}}-\frac{2(2t+n)}{B_{1}}\right\}}\right\}_{n\in\mathbb{Z}}$ form an ONB of $V^{M_{1},M_{2}}_{1}$ and also $\phi_{M_{1},M_{2},0,0}(t)=\phi(t)e^{-\frac{i}{2}\left\{\left(\frac{A_{1}}{B_{1}}-\frac{A_{2}}{B_{2}}\right)t^2-\frac{2t}{B_{1}}\right\}}\in V^{M_{1},M_{2}}_{0}\subset V^{M_{1},M_{2}}_{1}$ there exists $\{c_{n}\}_{n\in\mathbb{Z}}\in l^2(\mathbb{Z})$ for which 
$$\phi_{M_{1},M_{2},0,0}(t)=\sum_{n\in\mathbb{Z}}c_{n}\phi_{M_{1},M_{2},1,n}(t).$$
This implies 
$$\phi(t)e^{-\frac{i}{2}\left\{\left(\frac{A_{1}}{B_{1}}-\frac{A_{2}}{B_{2}}\right)t^2-\frac{2t}{B_{1}}\right\}}=\sum_{n\in\mathbb{Z}}c_{n}2^{\frac{1}{2}}\phi(2t-n)e^{-\frac{i}{2}\left\{\left(t^2-\left(\frac{n}{2}\right)^2\right)\frac{A_{1}}{B_{1}}-\left(2t-n\right)^2\frac{A_{2}}{B_{2}}-\frac{2(2t+n)}{B_{1}}\right\}}$$
\begin{align*}
\phi(t)e^{\frac{it}{B_{1}}}
&=\sum_{n\in\mathbb{Z}}c_{n}2^{\frac{1}{2}}\phi(2t-n)e^{-\frac{i}{2}\left\{-\left(\frac{n}{2}\right)^2\frac{A_{1}}{B_{1}}-\left(2t-n\right)^2\frac{A_{2}}{B_{2}}-\frac{2(2t+n)}{B_{1}}+\frac{A_{2}}{B_{2}}t^2\right\}}\\
&=\sum_{n\in\mathbb{Z}}c_{n}2^{\frac{1}{2}}\phi(2t-n)e^{\frac{i}{2}\left\{\left(\frac{n}{2}\right)^2\frac{A_{1}}{B_{1}}+\left(3t^2-4tn+n^2\right)\frac{A_{2}}{B_{2}}+\frac{2(2t+n)}{B_{1}}\right\}},
\end{align*}
where 
$$c_{n}=2^{\frac{1}{2}}e^{-\frac{i}{2}\left(\frac{n^2A_{1}}{4B_{1}}+\frac{n^2A_{2}}{B_{2}}+\frac{2n}{B_1}\right)}\int_{\mathbb{R}}\phi(t)\overline{\phi(2t-n)}e^{-\frac{i}{2}\left\{\left(3t^2-4tn\right)\frac{A_{2}}{B_{2}}+\frac{2t}{B_{1}}\right\}}dt.$$
\begin{align*}
\Big(\mathcal{L}^{M_{2}}&\left\{e^{\frac{it}{B_{1}}}\phi(t)\right\}\Big)(u)\\
&=2^{\frac{1}{2}}\sum_{n\in\mathbb{Z}}c_{n}\int_{\mathbb{R}}\phi(2t-n)e^{\frac{i}{2}\left\{\left(\frac{n}{2}\right)^2\frac{A_{1}}{B_{1}}+\left(3t^2-4tn+n^2\right)\frac{A_{2}}{B_{2}}+\frac{2(2t+n)}{B_{1}}\right\}}K_{M_{2}}(t,u)dt\\
&=\frac{2^{\frac{1}{2}}}{\sqrt{2\pi iB_{2}}}\sum_{n\in\mathbb{Z}}c_{n}\int_{\mathbb{R}}\phi(2t-n)e^{\frac{i}{2}\left\{\frac{n^2A_{1}}{4B_{1}}+\left(3t^2-4tn+n^2\right)\frac{A_{2}}{B_{2}}+\frac{2(2t+n)}{B_{1}}+\left(\frac{A_{1}}{B_{1}}t^2-\frac{2tu}{B_{2}}+\frac{D_{2}}{B_{2}}u^2\right)\right\}}dt\\
&=\frac{2^{\frac{1}{2}}}{\sqrt{2\pi iB_{2}}}\sum_{n\in\mathbb{Z}}c_{n}\int_{\mathbb{R}}e^{\frac{i(2t-n)}{B_{1}}}\phi(2t-n)e^{\frac{i}{2}\left(\frac{n^2A_{1}}{4B_{1}}\right)+\frac{2in}{B_{1}}}  e^{\frac{i}{2}\left\{(2t-n)^2\frac{A_{2}}{B_{2}}-\frac{2tu}{B_{2}}+\frac{D_{2}}{B_{2}}u^2\right\}}dt\\
&=\frac{2^{\frac{1}{2}}}{\sqrt{2\pi iB_{2}}}\sum_{n\in\mathbb{Z}}c_{n}\int_{\mathbb{R}}e^{\frac{i(2t-n)}{B_{1}}}\phi(2t-n)e^{\frac{i}{2}\left(\frac{n^2A_{1}}{4B_{1}}\right)+\frac{2in}{B_{1}}}  e^{\frac{i}{2}\left\{(2t-n)^2\frac{A_{2}}{B_{2}}-\frac{2(2t-n)}{B_{2}}\left(\frac{u}{2}\right)+\frac{D_{2}}{B_{2}}\left(\frac{u}{2}\right)^2\right\}}e^{\frac{i}{2}\left\{-\frac{nu}{B_{2}}+\frac{3D_{2}}{B_{2}}\left(\frac{u}{2}\right)^2\right\}}dt\\
&=\frac{2^{\frac{1}{2}}}{\sqrt{2\pi iB_{2}}}\sum_{n\in\mathbb{Z}}\frac{c_{n}}{2}\int_{\mathbb{R}}e^{\frac{it}{B_{1}}}\phi(t)e^{\frac{i}{2}\left(\frac{n^2A_{1}}{4B_{1}}\right)+\frac{2in}{B_{1}}}e^{\frac{i}{2}\left\{\frac{A_{2}}{B_{2}}t^2-\frac{2t}{B_{2}}\left(\frac{u}{2}\right)+\frac{D_{2}}{B_{2}}\left(\frac{u}{2}\right)^2\right\}}e^{\frac{i}{2}\left\{-\frac{nu}{B_{2}}+\frac{3D_{2}}{B_{2}}\left(\frac{u}{2}\right)^2\right\}}dt\\
&=\frac{2^{\frac{1}{2}}}{2}\sum_{n\in\mathbb{Z}}c_{n}\int_{\mathbb{R}}e^{\frac{it}{B_{1}}}\phi(t)e^{\frac{i}{2}\left\{\frac{n^2A_{1}}{4B_{1}}+\frac{4n}{B_{1}}-\frac{nu}{B_{3}}+\frac{3D_{2}}{B_{2}}\left(\frac{u}{2}\right)^2\right\}}K_{M_{2}}\left(t,\frac{u}{2}\right)dt\\
&=\frac{2^{\frac{1}{2}}}{2}\sum_{n\in\mathbb{Z}}c_{n}e^{\frac{i}{2}\left\{\frac{n^2A_{1}}{4B_{1}}+\frac{4n}{B_{1}}-\frac{nu}{B_{3}}+\frac{3D_{2}}{B_{2}}\left(\frac{u}{2}\right)^2\right\}}\left(\mathcal{L}^{M_{2}}\left\{e^{\frac{it}{B_{1}}}\phi(t)\right\}\right)\left(\frac{u}{2}\right)\\
&=e^{\frac{i}{2}\left(\frac{3u^2D_{2}}{4B_{2}}\right)}\bigwedge^{M_{1},M_{2}}\left(\frac{u}{2}\right)\left(\mathcal{L}^{M_{2}}\left\{e^{\frac{it}{B_{1}}}\phi(t)\right\}\right)\left(\frac{u}{2}\right),
\end{align*}
where
\begin{align}\label{P4Construction-WaveletEqn1}
\bigwedge^{M_{1},M_{2}}(u)=\frac{\sqrt{2}}{2}\sum_{n\in\mathbb{Z}}c_{n}e^{\frac{i}{2}\left(\frac{n^2A_{1}}{4B_{1}}+\frac{4n}{B_{1}}-\frac{2nu}{B_{2}}\right)}.
\end{align}
Therefore,
$$\left(\mathcal{L}^{M_{2}}\left\{e^{\frac{it}{B_{1}}}\phi(t)\right\}\right)\left(\frac{B_{2}u}{B_{1}}\right)=e^{\frac{i}{2}\left(\frac{3D_{2}}{4B_{2}}\right)\left(\frac{B_{2}u}{B_{1}}\right)^2}\bigwedge^{M_{1},M_{2}}\left(\frac{B_{2}u}{2B_{1}}\right)\left(\mathcal{L}^{M_{2}}\left\{e^{\frac{it}{B_{1}}}\phi(t)\right\}\right)\left(\frac{B_{2}u}{2B_{1}}\right)$$
Since $\left\{\phi_{M_{1},M_{2},0,n}(t)=\phi(t-n)e^{-\frac{i}{2}\left\{(t^2-n^2)\frac{A_{1}}{B_{1}}-(t-n)^2\frac{A_{2}}{B_{2}}-\frac{2(t+n)}{B_{1}}\right\}}\right\}_{n\in\mathbb{Z}}$ is an ONB of $V^{M_{1},M_{2}}_{0},$ so we have 
$$\sum_{k\in\mathbb{Z}}\left|\left(\mathcal{L}^{M_{2}}\left\{e^{\frac{it}{B_{1}}}\phi(t)\right\}\right)\left(\frac{B_{2}}{B_{1}}\left(u+2k\pi|B_{1}|\right)\right)\right|^2=\frac{1}{2\pi|B_{2}|}.$$
Moreover,
\begin{align}\label{P4Construction-WaveletEqn2}
\sum_{k\in\mathbb{Z}}&\left|\left(\mathcal{L}^{M_{2}}\{e^{\frac{it}{B_{1}}}\phi(t)\}\right)\left(\frac{B_{2}}{B_{1}}\left(u+2k\pi|B_{1}|\right)\right)\right|^2\notag\\
&=\sum_{k\in\mathbb{Z}}\left|\bigwedge^{M_{1},M_{2}}\left(\frac{B_{2}}{2B_{1}}\left(u+2k\pi|B_{1}|\right)\right)\right|^2\left|\left(\mathcal{L}^{M_{2}}\left\{e^{\frac{it}{B_{1}}}\phi(t)\right\}\right)\left(\frac{B_{2}}{2B_{1}}(u+2k\pi|B_{1}|)\right)\right|^2\notag\\
&=\sum_{l\in\mathbb{Z}}\left|\bigwedge^{M_{1},M_{2}}\left(\frac{B_{2}}{2B_{1}}\left(u+4l\pi|B_{1}|\right)\right)\right|^2\left|\left(\mathcal{L}^{M_{2}}\left\{e^{\frac{it}{B_{1}}}\phi(t)\right\}\right)\left(\frac{B_{2}}{2B_{1}}\left(u+4l\pi|B_{1}|\right)\right)\right|^2\notag\\
&\hspace{1cm}+\sum_{l\in\mathbb{Z}}\left|\bigwedge^{M_{1},M_{2}}\left(\frac{B_{2}}{2B_{1}}\left(u+(4l+2)\pi|B_{1}|\right)\right)\right|^2\left|\left(\mathcal{L}^{M_{2}}\left\{e^{\frac{it}{B_{1}}}\phi(t)\right\}\right)\left(\frac{B_{2}}{2B_{1}}\left(u+(4l+2)\pi|B_{1}|\right)\right)\right|^2\notag\\
&=\sum_{l\in\mathbb{Z}}\left|\bigwedge^{M_{1},M_{2}}\left(\frac{B_{2}}{B_{1}}\left(\frac{u}{2}+2l\pi|B_{1}|\right)\right)\right|^2\left|\left(\mathcal{L}^{M_{2}}\left\{e^{\frac{it}{B_{1}}}\phi(t)\right\}\right)\left(\frac{B_{2}}{B_{1}}\left(\frac{u}{2}+2l\pi|B_{1}|\right)\right)\right|^2\notag\\
&\hspace{1cm}+\sum_{l\in\mathbb{Z}}\left|\bigwedge^{M_{1},M_{2}}\left(\frac{B_{2}}{B_{1}}\left(\frac{u}{2}+(2l+1)\pi|B_{1}|\right)\right)\right|^2\left|\left(\mathcal{L}^{M_{2}}\left\{e^{\frac{it}{B_{1}}}\phi(t)\right\}\right)\left(\frac{B_{2}}{B_{1}}\left(\frac{u}{2}+(2l+1)\pi|B_{1}|\right)\right)\right|^2
\end{align}
Since $\bigwedge^{M_{1},M_{2}}\left(\frac{B_{2}u}{B_{1}}\right)$ is $2\pi|B_{1}|$ periodic, from \eqref{P4Construction-WaveletEqn2}, we have
\begin{align*}
\sum_{k\in\mathbb{Z}}&\left|\left(\mathcal{L}^{M_{2}}\left\{e^{\frac{it}{B_{1}}}\phi(t)\right\}\right)\left(\frac{B_{2}}{B_{1}}\left(u+2k\pi|B_{1}|\right)\right)\right|^2\\
&=\sum_{l\in\mathbb{Z}}\left|\bigwedge^{M_{1},M_{2}}\left(\frac{B_{2}u}{2B_{1}}\right)\right|^2\left|\left(\mathcal{L}^{M_{2}}\left\{e^{\frac{it}{B_{1}}}\phi(t)\right\}\right)\left(\frac{B_{2}}{B_{1}}(\frac{u}{2}+2l\pi|B_{1}|)\right)\right|^2\\
&\hspace{1cm}+\sum_{l\in\mathbb{Z}}\left|\bigwedge^{M_{1},M_{2}}\left(\frac{B_{2}}{B_{1}}\left(\frac{u}{2}+\pi|B_{1}|\right)\right)\right|^2\left|\left(\mathcal{L}^{M_{2}}\left\{e^{\frac{it}{B_{1}}}\phi(t)\right\}\right)\left(\frac{B_{2}}{B_{1}}\left(\frac{u}{2}+(2l+1)\pi|B_{1}|\right)\right)\right|^2\\
&=\left|\bigwedge^{M_{1},M_{2}}\left(\frac{B_{2}u}{2B_{1}}\right)\right|^2\sum_{l\in\mathbb{Z}}\left|\left(\mathcal{L}^{M_{2}}\left\{e^{\frac{it}{B_{1}}}\phi(t)\right\}\right)\left(\frac{B_{2}}{B_{1}}(\frac{u}{2}+2l\pi|B_{1}|)\right)\right|^2\\
&\hspace{1cm}+\left|\bigwedge^{M_{1},M_{2}}\left(\frac{B_{2}}{B_{1}}\left(\frac{u}{2}+\pi|B_{1}|\right)\right)\right|^2\sum_{l\in\mathbb{Z}}\left|\left(\mathcal{L}^{M_{2}}\left\{e^{\frac{it}{B_{1}}}\phi(t)\right\}\right)\left(\frac{B_{2}}{B_{1}}\left(\frac{u}{2}+(2l+1)\pi|B_{1}|\right)\right)\right|^2.
\end{align*}
Since 
$$\sum_{k\in\mathbb{Z}}\left|\left(\mathcal{L}^{M_{2}}\left\{e^{\frac{it}{B_{1}}}\phi(t)\right\}\right)\left(\frac{B_{2}}{B_{1}}\left(u+2k\pi|B_{1}|\right)\right)\right|^2=\frac{1}{2\pi|B_{2}|},$$
$$\sum_{l\in\mathbb{Z}}\left|\left(\mathcal{L}^{M_{2}}\left\{e^{\frac{it}{B_{1}}}\phi(t)\right\}\right)\left(\frac{B_{2}}{B_{1}}\left(\frac{u}{2}+2l\pi|B_{1}|\right)\right)\right|^2=\frac{1}{2\pi|B_{2}|}$$
and
$$\sum_{l\in\mathbb{Z}}\left|\left(\mathcal{L}^{M_{2}}\left\{e^{\frac{it}{B_{1}}}\phi(t)\right\}\right)\left(\frac{B_{2}}{B_{1}}\left(\frac{u}{2}+(2l+1)\pi|B_{1}|\right)\right)\right|^2=\frac{1}{2\pi|B_{2}|}.$$
Therefore, we get
$$\left|\bigwedge^{M_{1},M_{2}}\left(\frac{B_{2}u}{2B_{1}}\right)\right|^2+\left|\bigwedge^{M_{1},M_{2}}\left(\frac{B_{2}}{B_{1}}\left(\frac{u}{2}+\pi|B_{1}|\right)\right)\right|^2=1.$$
Equivalently, we have
$$\left|\bigwedge^{M_{1},M_{2}}\left(\frac{B_{2}u}{B_{1}}\right)\right|^2+\left|\bigwedge^{M_{1},M_{2}}\left(\frac{B_{2}}{B_{1}}\left(u+\pi|B_{1}|\right)\right)\right|^2=1.$$
Since $\psi_{M_{1},M_{2},0,0}(t)=\psi(t)e^{-\frac{i}{2}\left\{\left(\frac{A_{1}}{B_{1}}-\frac{A_{2}}{B_{2}}\right)t^2-\frac{2t}{B_{1}}\right\}}\in W^{M_{1},M_{2}}_{0}\subset V^{M_{1},M_{2}}_{1},$ there exists $\{d_{n}\}\in l^2(\mathbb{Z})$ such that 
\begin{align}\label{P4WaveletEquation}
\psi_{M_{1},M_{2},0,0}(t)=\sum_{n\in\mathbb{Z}}d_{n}\phi_{M_{1},M_{2},1,n}(t).
\end{align}
On equation \eqref{P4WaveletEquation} we take $\mathcal{L}^{M_{2}}$ and proceed as above to get 
$$\left|\Gamma^{M_{1},M_{2}}\left(\frac{B_{2}u}{B_{1}}\right)\right|^2+\left|\Gamma^{M_{1},M_{2}}\left(\frac{B_{2}}{B_{1}}\left(u+\pi|B_{1}|\right)\right)\right|^2=1,$$
where 
\begin{align}\label{P4Construction-WaveletEqn3}
\displaystyle\Gamma^{M_{1},M_{2}}(u)=\frac{\sqrt{2}}{2}\sum_{n\in\mathbb{Z}}d_{n}e^{\frac{i}{2}\left(\frac{n^2A_{1}}{4B_{1}}+\frac{4n}{B_{1}}-\frac{2nu}{B_{2}}\right)}.
\end{align}
Moreover, since $W^{M_{1},M_{2}}_{0}$ and $V^{M_{1},M_{2}}_{0}$ are orthogonal in $V^{M_{1},M_{2}}_{1}$ so 
$$\langle\phi_{M_{1},M_{2},0,m},\psi_{M_{1},M_{2},0,n} \rangle_{L^2(\mathbb{R})}=0,~\forall m,n\in\mathbb{Z}.$$
This gives 
$$\langle\mathcal{L}^{M_{1}}\phi_{M_{1},M_{2},0,m},\mathcal{L}^{M_{1}}\psi_{M_{1},M_{2},0,n} \rangle_{L^2(\mathbb{R})}=0,~\forall m,n\in\mathbb{Z}.$$
Now
\begin{align}\label{P4Construction-WaveletEqn4}
\langle&\mathcal{L}^{M_{1}}\phi_{M_{1},M_{2},0,m},\mathcal{L}^{M_{1}}\psi_{M_{1},M_{2},0,n}\rangle_{L^2(\mathbb{R})}\notag\\
&=2\pi|B_{2}|\int_{\mathbb{R}}e^{\frac{2i}{B_{2}}(n-m)}K_{M_{1}}(n,u)\overline{K_{M_{1}}(m,u)}\left(\mathcal{L}^{M_{1}}\left\{e^{\frac{it}{B_{1}}}\phi(t)\right\}\right)\left(\frac{B_{2}u}{B_{1}}\right)\overline{\left(\mathcal{L}^{M_{1}}\{e^{\frac{it}{B_{1}}}\psi(t)\}\right)\left(\frac{B_{2}u}{B_{1}}\right)}du\notag\\
&=2\pi|B_{2}|e^{\frac{2i}{B_{2}}(n-m)}\int_{\mathbb{R}}\bigwedge^{M_{1},M_{2}}\left(\frac{B_{2}u}{2B_{1}}\right)\overline{\Gamma^{M_{1},M_{2}}\left(\frac{B_{2}u}{2B_{1}}\right)}\left|\left(\mathcal{L}^{M_{2}}\left\{e^{\frac{it}{B_{1}}}\phi(t)\right\}\right)\left(\frac{B_{2}u}{2B_{1}}\right)\right|^2K_{M_{1}}(n,u)\overline{K_{M_{1}}(m,u)}du\notag\\
&=2\pi|B_{2}|e^{\frac{2i}{B_{2}}(n-m)}\sum_{k\in\mathbb{Z}}\int_{4\pi|B_{1}|}^{4(k+1)\pi|B_{1}|}\bigwedge^{M_{1},M_{2}}\left(\frac{B_{2}u}{2B_{1}}\right)\overline{\Gamma^{M_{1},M_{2}}\left(\frac{B_{2}u}{2B_{1}}\right)}\left|\left(\mathcal{L}^{M_{2}}\left\{e^{\frac{it}{B_{1}}}\phi(t)\right\}\right)\left(\frac{B_{2}u}{2B_{1}}\right)\right|^2\notag\\
&\hspace{11cm}\times K_{M_{1}}(n,u)\overline{K_{M_{1}}(m,u)}du\notag\\
&=2\pi|B_{2}|e^{\frac{2i}{B_{2}}(n-m)}\sum_{k\in\mathbb{Z}}\int_{0}^{4\pi|B_{1}|}\bigwedge^{M_{1},M_{2}}\left(\frac{B_{2}}{2B_{1}}\left(u+4k\pi|B_{1}|\right)\right)\overline{\Gamma^{M_{1},M_{2}}\left(\frac{B_{2}}{2B_{1}}\left(u+4k\pi|B_{1}|\right)\right)}\notag\\
&\hspace{3.5cm}\times \left|\left(\mathcal{L}^{M_{2}}\left\{e^{\frac{it}{B_{1}}}\phi(t)\right\}\right)\left(\frac{B_{2}}{2B_{1}}\left(u+4k\pi|B_{1}|\right)\right)\right|^2K_{M_{1}}(n,u+4k\pi|B_{1}|)\overline{K_{M_{1}}(m,u+4k\pi|B_{1}|)}du\notag\\
&=2\pi|B_{2}|e^{\frac{2i}{B_{2}}(n-m)}\sum_{k\in\mathbb{Z}}\int_{0}^{4\pi|B_{1}|}\bigwedge^{M_{1},M_{2}}\left(\frac{B_{2}u}{2B_{1}}\right)\overline{\Gamma^{M_{1},M_{2}}\left(\frac{B_{2}u}{2B_{1}}\right)}\left|\left(\mathcal{L}^{M_{2}}\left\{e^{\frac{it}{B_{1}}}\phi(t)\right\}\right)\left(\frac{B_{2}}{2B_{1}}\left(u+4k\pi|B_{1}|\right)\right)\right|^2\notag\\
&\hspace{9cm}\times K_{M_{1}}(n,u+4k\pi|B_{1}|)\overline{K_{M_{1}}(m,u+4k\pi|B_{1}|)}du,
\end{align}
by $2\pi|B_{1}|$ periodicity of $\bigwedge^{M_{1},M_{2}}\left(\frac{B_{2}u}{B_{1}}\right)~\mbox{and}~\Gamma^{M_{1},M_{2}}\left(\frac{B_{2}u}{B_{1}}\right).$\\
Since $K_{M_{1}}(n,u+4k\pi|B_{1}|)\overline{K_{M_{1}}(m,u+4k\pi|B_{1}|)}=K_{M_{1}}(n,u)\overline{K_{M_{1}}(m,u)},$ from \eqref{P4Construction-WaveletEqn4} we have
\begin{align*}
&\left\langle\mathcal{L}^{M_{1}}\phi_{M_{1},M_{2},0,m},\mathcal{L}^{M_{1}}\psi_{M_{1},M_{2},0,n}\right\rangle_{L^2(\mathbb{R})}\\
&=2\pi|B_{2}|e^{\frac{2i}{B_{2}}(n-m)}\sum_{k\in\mathbb{Z}}\int_{0}^{4\pi|B_{1}|}\bigwedge^{M_{1},M_{2}}\left(\frac{B_{2}u}{2B_{1}}\right)\overline{\Gamma^{M_{1},M_{2}}\left(\frac{B_{2}u}{2B_{1}}\right)}\left|\left(\mathcal{L}^{M_{2}}\{e^{\frac{it}{B_{1}}}\phi(t)\}\right)\left(\frac{B_{2}}{2B_{1}}\left(u+4k\pi|B_{1}|\right)\right)\right|^2\notag\\
&\hspace{13cm}\times K_{M_{1}}(m,u)\overline{K_{M_{1}}(n,u)}du\\
&=2\pi|B_{2}|e^{\frac{2i}{B_{2}}(n-m)}\int_{0}^{4\pi|B_{1}|}\bigwedge^{M_{1},M_{2}}\left(\frac{B_{2}u}{2B_{1}}\right)\overline{\Gamma^{M_{1},M_{2}}\left(\frac{B_{2}u}{2B_{1}}\right)}\left\{\sum_{k\in\mathbb{Z}}\left|\left(\mathcal{L}^{M_{2}}\left\{e^{\frac{it}{B_{1}}}\phi(t)\right\}\right)\left(\frac{B_{2}}{2B_{1}}\left(u+4k\pi|B_{1}|\right)\right)\right|^2\right\}\notag\\
&\hspace{13cm}\times K_{M_{1}}(m,u)\overline{K_{M_{1}}(n,u)}du\\
&=\frac{1}{2\pi|B_{2}|}2\pi|B_{2}|e^{\frac{2i}{B_{2}}(n-m)}\int_{0}^{4\pi|B_{1}|}\bigwedge^{M_{1},M_{2}}\left(\frac{B_{2}u}{2B_{1}}\right)\overline{\Gamma^{M_{1},M_{2}}\left(\frac{B_{2}u}{2B_{1}}\right)}K_{M_{1}}(n,u)\overline{K_{M_{1}}(m,u)}du\\
&=e^{\frac{2i}{B_{2}}(n-m)}\Bigg\{\int_{0}^{2\pi|B_{1}|}\bigwedge^{M_{1},M_{2}}\left(\frac{B_{2}u}{2B_{1}}\right)\overline{\Gamma^{M_{1},M_{2}}\left(\frac{B_{2}u}{2B_{1}}\right)}K_{M_{1}}(m,u)\overline{K_{M_{1}}(n,u)}du\\
&\hspace{7cm}+\int_{2\pi|B_{1}|}^{4\pi|B_{1}|}\bigwedge^{M_{1},M_{2}}\left(\frac{B_{2}u}{2B_{1}}\right)\overline{\Gamma^{M_{1},M_{2}}(\frac{B_{2}u}{2B_{1}})}K_{M_{1}}(n,u)\overline{K_{M_{1}}(m,u)}du\Bigg\}.
\end{align*}
Now
\begin{align*}
\int_{0}^{2\pi|B_{1}|}\bigwedge^{M_{1},M_{2}}\left(\frac{B_{2}u}{2B_{1}}\right)&\overline{\Gamma^{M_{1},M_{2}}\left(\frac{B_{2}u}{2B_{1}}\right)}K_{M_{1}}(n,u)\overline{K_{M_{1}}(m,u)}du\\
&=\frac{1}{\left(\sqrt{2\pi iB_{1}}\right)^2}e^{\frac{iA_{1}}{2B_{1}}\left(m^2-n^2\right)}\int_{0}^{2\pi|B_{1}|}\bigwedge^{M_{1},M_{2}}\left(\frac{B_{2}u}{2B_{1}}\right)\overline{\Gamma^{M_{1},M_{2}}\left(\frac{B_{2}u}{2B_{1}}\right)}e^{-\frac{iu}{B_{1}}(m-n)}du
\end{align*}
and 
\begin{align*}
\int_{2\pi|B_{1}|}^{4\pi|B_{1}|}&\bigwedge^{M_{1},M_{2}}\left(\frac{B_{2}u}{2B_{1}}\right)\overline{\Gamma^{M_{1},M_{2}}\left(\frac{B_{2}u}{2B_{1}}\right)}K_{M_{1}}(n,u)\overline{K_{M_{1}}(m,u)}du\\
&=\frac{1}{\left(\sqrt{2\pi iB_{1}}\right)^2}\int_{0}^{2\pi|B_{1}|}\bigwedge^{M_{1},M_{2}}\left(\frac{B_{2}}{2B_{1}}\left(u+2\pi|B_{1}|\right)\right)\overline{\Gamma^{M_{1},M_{2}}\left(\frac{B_{2}}{2B_{1}}\left(u+2\pi|B_{1}|\right)\right)}\\
&\hspace{7cm}\times K_{M_{1}}\left(n,(u+2\pi|B_{1}|)\right)\overline{K_{M_{1}}\left(m,(u+2\pi|B_{1}|)\right)}du\\
&=\frac{1}{\left(\sqrt{2\pi iB_{1}}\right)^2}\int_{0}^{2\pi|B_{1}|}\bigwedge^{M_{1},M_{2}}\left(\frac{B_{2}}{2B_{1}}\left(u+2\pi|B_{1}|\right)\right)\overline{\Gamma^{M_{1},M_{2}}\left(\frac{B_{2}}{2B_{1}}\left(u+2\pi|B_{1}|\right)\right)}e^{\frac{i}{2}\left\{\frac{A_{1}}{B_{1}}(m^2-n^2)-\frac{2u}{B_{1}}(m-n)\right\}}du\\
&=\frac{1}{\left(\sqrt{2\pi iB_{1}}\right)^2}e^{\frac{iA_{1}}{2B_{1}}(m^2-n^2)}\int_{0}^{2\pi|B_{1}|}\bigwedge^{M_{1},M_{2}}\left(\frac{B_{2}}{2B_{1}}\left(u+2\pi|B_{1}|\right)\right)\overline{\Gamma^{M_{1},M_{2}}\left(\frac{B_{2}}{2B_{1}}\left(u+2\pi|B_{1}|\right)\right)}e^{-\frac{iu}{B_{1}}(m-n)}du.
\end{align*}
Therefore,
\begin{align*}
&\left\langle\mathcal{L}^{M_{1}}\phi_{M_{1},M_{2},0,m},\mathcal{L}^{M_{1}}\psi_{M_{1},M_{2},0,n} \right\rangle_{L^2(\mathbb{R})}\\
&=\frac{1}{\left(\sqrt{2\pi iB_{1}}\right)^2}e^{\frac{2i}{B_{2}}(n-m)}e^{\frac{iA_{1}}{2B_{1}}(m^2-n^2)}\int_{0}^{2\pi|B_{1}|}\Bigg\{\bigwedge^{M_{1},M_{2}}\left(\frac{B_{2}u}{2B_{1}}\right)\overline{\Gamma^{M_{1},M_{2}}\left(\frac{B_{2}u}{2B_{1}}\right)}\\&\hspace{6cm}+\bigwedge^{M_{1},M_{2}}\left(\frac{B_{2}}{2B_{1}}\left(u+2\pi|B_{1}|\right)\right)\overline{\Gamma^{M_{1},M_{2}}\left(\frac{B_{2}}{2B_{1}}\left(u+2\pi|B_{1}|\right)\right)}\Bigg\}e^{-\frac{iu}{B_{1}}(m-n)}du.
\end{align*}
Since $\left\langle\mathcal{L}^{M_{1}}\phi_{M_{1},M_{2},0,m},\mathcal{L}^{M_{1}}\psi_{M_{1},M_{2},0,n}\right\rangle_{L^2(\mathbb{R})}=0,$ it follows that 
$$\int_{0}^{2\pi|B_{1}|}\left\{\bigwedge^{M_{1},M_{2}}\left(\frac{B_{2}u}{2B_{1}}\right)\overline{\Gamma^{M_{1},M_{2}}\left(\frac{B_{2}u}{2B_{1}}\right)}+\bigwedge^{M_{1},M_{2}}\left(\frac{B_{2}}{2B_{1}}\left(u+2\pi|B_{1}|\right)\right)\overline{\Gamma^{M_{1},M_{2}}\left(\frac{B_{2}}{2B_{1}}\left(u+2\pi|B_{1}|\right)\right)}\right\}e^{\frac{i(m-n)u}{B_{1}}}du=0.$$
Since $\left\{\frac{1}{\sqrt{2\pi}}e^{-in\left(\frac{u}{B_{1}}\right)}\right\}_{n\in\mathbb{Z}}$ is an ONB of $L^2[0,2\pi|B_{1}|],$ we have 
$$\bigwedge^{M_{1},M_{2}}\left(\frac{B_{2}u}{B_{1}}\right)\overline{\Gamma^{M_{1},M_{2}}\left(\frac{B_{2}u}{B_{1}}\right)}+\bigwedge^{M_{1},M_{2}}\left(\frac{B_{2}}{B_{1}}\left(u+2\pi|B_{1}|\right)\right)\overline{\Gamma^{M_{1},M_{2}}\left(\frac{B_{2}}{B_{1}}\left(u+2\pi|B_{1}|\right)\right)}=0,~u\in\mathbb{R}.$$
Thus in view of the above discussion we have the following important theorem.
\begin{theorem}
If  $\displaystyle\psi(t)=\sum_{n\in\mathbb{Z}}d_{n}\sqrt{2}e^{\frac{i}{2}\left\{\frac{n^2A_1}{4B_1}+\left(3t^2-4tn+n^2\right)\frac{A_2}{B_2}+\frac{2(t+n)}{B_1}\right\}},$ then $\{\psi_{M_{1},M_{2},0,n}\}_{n\in\mathbb{Z}}$ is an ONB of $W^{M_{1},M_{2}}_{0}$ if and only if
\begin{align}\label{P4IffforWavelet}
U\left(\frac{B_{2}u}{B_{1}}\right)U^{\star}\left(\frac{B_{2}u}{B_{1}}\right)=I,
\end{align}
where 
$I$ is the identity matrix and $U^{\star}\left(\frac{B_{2}u}{B_{1}}\right)$ denotes the conjugate transpose of the matrix $U\left(\frac{B_{2}u}{B_{1}}\right),$ where
$$U\left(\frac{B_{2}u}{B_{1}}\right)=
\begin{bmatrix}
\bigwedge^{M_{1},M_{2}}\left(\frac{B_{2}u}{B_{1}}\right) & \bigwedge^{M_{1},M_{2}}\left(\frac{B_{2}}{B_{1}}\left(u+\pi|B_{1}|\right)\right)\\
\Gamma^{M_{1},M_{2}}\left(\frac{B_{2}u}{B_{1}}\right)& \Gamma^{M_{1},M_{2}}\left(\frac{B_{2}}{B_{1}}\left(u+\pi|B_{1}|\right)\right)
\end{bmatrix}.$$
\end{theorem}
If \eqref{P4IffforWavelet} holds, then $\exists$ a function $\rho\left(\frac{B_{2}u}{B_{1}}\right)$ which satisfies
$$\left(\Gamma^{M_{1},M_{2}}\left(\frac{B_{2}u}{B_{1}}\right),\Gamma^{M_{1},M_{2}}\left(\frac{B_{2}}{B_{1}}\left(u+\pi|B_{1}|\right)\right)\right)=\left( \rho\left(\frac{B_{2}u}{B_{1}}\right)\overline{\bigwedge^{M_{1},M_{2}}\left(\frac{B_{2}}{B_{1}}\left(u+\pi|B_{1}|\right)\right)},-\rho\left(\frac{B_{2}u}{B_{1}}\right)\overline{\bigwedge^{M_{1},M_{2}}\left(\frac{B_{2}u}{B_{1}}\right)}\right).$$
Since 
$\rho\left(\frac{B_{2}u}{B_{1}}\right)$ is $2\pi|B_{1}|-$periodic function, therefore we have
\begin{align}\label{P4DiscussionEqn1}
\rho\left(\frac{B_{2}u}{B_{1}}\right)=\sum_{k\in\mathbb{Z}}c_{k}e^{-\frac{iku}{B_{1}}},
\end{align}
where
\begin{align}\label{P4DiscussionEqn2}
c_{k}&=\frac{1}{2\pi|B_{1}|}\int_{0}^{2\pi|B_{1}|}\rho\left(\frac{B_{2}u}{B_{1}}\right)e^{\frac{iku}{B_{1}}}du\notag\\
&=\frac{1}{2\pi|B_{1}|}\int_{0}^{\pi|B_{1}|}\rho\left(\frac{B_{2}u}{B_{1}}\right)e^{\frac{iku}{B_{1}}}du+\frac{1}{2\pi|B_{1}|}\int_{0}^{\pi|B_{1}|}\rho\left(\frac{B_{2}}{B_{1}}\left(u+\pi|B_{1}|\right)\right)e^{\frac{ik}{B_{1}}\left(u+\pi|B_{1}|\right)}du\notag\\
&=\frac{1}{2\pi|B_{1}|}\int_{0}^{\pi|B_{1}|}\rho\left(\frac{B_{2}u}{B_{1}}\right)e^{\frac{iku}{B_{1}}}du+(-1)^k\frac{1}{2\pi|B_{1}|}\int_{0}^{\pi|B_{1}|}\rho\left(\frac{B_{2}}{B_{1}}\left(u+\pi|B_{1}|\right)\right)e^{\frac{iku}{B_{1}}}du.
\end{align}
Since $\Gamma^{M_{1},M_{2}}\left(\frac{B_{2}u}{B_{1}}\right)$ is $2\pi|B_{1}|$ periodic, we have
\begin{align}\label{P4DiscussionEqn3}
\Gamma^{M_{1},M_{2}}\left(\frac{B_{2}}{B_{1}}\left(u+\pi|B_{1}|\right)\right)=\rho\left(\frac{B_{2}}{B_{1}}\left(u+\pi|B_{1}|\right)\right)\overline{\bigwedge^{M_{1},M_{2}}\left(\frac{B_{2}u}{B_{1}}\right)}.
\end{align}
Also,
\begin{align}\label{P4DiscussionEqn4}
\Gamma^{M_{1},M_{2}}\left(\frac{B_{2}}{B_{1}}\left(u+\pi|B_{1}|\right)\right)=\rho\left(\frac{B_{2}u}{B_{1}}\right)\overline{\bigwedge^{M_{1},M_{2}}\left(\frac{B_{2}u}{B_{1}}\right)}.
\end{align}
Therefore, from \eqref{P4DiscussionEqn3} and \eqref{P4DiscussionEqn4} we have
\begin{align}\label{P4DiscussionEqn5}
\rho\left(\frac{B_{2}}{B_{1}}\left(u+\pi|B_{1}|\right)\right)=-\rho\left(\frac{B_{2}u}{B_{1}}\right).
\end{align}
Using \eqref{P4DiscussionEqn5} in \eqref{P4DiscussionEqn2}, we have
$$c_{k}=\frac{1}{2\pi|B_{1}|}\{1-(-1)^k\}\int_{0}^{\pi|B_{1}|}\rho\left(\frac{B_{2}u}{B_{1}}\right)e^{\frac{iku}{B_{1}}}du.$$
This gives 
\begin{align}\label{P4DiscussionEqn6}
c_{k}=0,\forall k=2m,~m\in\mathbb{Z}.
\end{align}
Therefore, using \eqref{P4DiscussionEqn6} in \eqref{P4DiscussionEqn1} gives
\begin{align*}
\rho\left(\frac{B_{2}u}{B_{1}}\right)
&=\sum_{l\in\mathbb{Z}}c_{2l+1}e^{-\frac{i(2l+1)u}{B_{1}}}\\
&=e^{-\frac{iu}{B_{1}}}\gamma\left(\frac{2B_{2}u}{B_{1}}\right),~\mbox{where}~\gamma\left(\frac{B_{2}u}{B_{1}}\right)=\sum_{l\in\mathbb{Z}}c_{2l+1}e^{-\frac{ilu}{B_{1}}}.
\end{align*}
Moreover, using \eqref{P4Construction-WaveletEqn1}, we get
\begin{align}\label{P4Construction-WaveletEqn7}
\overline{\bigwedge^{M_{1},M_{2}}\left(\frac{B_{2}}{B_{1}}\left(u+\pi|B_{1}|\right)\right)}
&=\frac{\sqrt{2}}{2}\sum_{n\in\mathbb{Z}}\overline{c_{n}}e^{-\frac{i}{2}\left\{\frac{n^2A_{1}}{4B_{1}}+\frac{4n}{B_{1}}-\frac{2n}{B_{1}}\left(u+\pi|B_{1}|\right)\right\}}\notag\\
&=\frac{\sqrt{2}}{2}\sum_{n\in\mathbb{Z}}(-1)^n\overline{c_{n}}e^{-\frac{i}{2}\left(\frac{n^2A_{1}}{4B_{1}}+\frac{4n}{B_{1}}-\frac{2nu}{B_{1}}\right)}.
\end{align}
Also, we have
\begin{align}\label{P4Construction-WaveletEqn8}
\Gamma^{M_{1},M_{2}}\left(\frac{B_{2}u}{B_{1}}\right)=\frac{\sqrt{2}}{2}\sum_{n\in\mathbb{Z}}d_{n}e^{\frac{i}{2}\left(\frac{n^2A_{1}}{4B_{1}}+\frac{4n}{B_{1}}-\frac{2nu}{B_{1}}\right)}.
\end{align}
Now, 
\begin{align*}
\Gamma^{M_{1},M_{2}}\left(\frac{B_{2}u}{B_{1}}\right)
&=\rho\left(\frac{B_{2}u}{B_{1}}\right)\overline{\bigwedge^{M_{1},M_{2}}\left(\frac{B_{2}}{B_{1}}\left(u+\pi|B_{1}|\right)\right)}\\
&=e^{-\frac{iu}{B_{1}}}\gamma\left(\frac{2B_{2}u}{B_{1}}\right)\overline{\bigwedge^{M_{1},M_{2}}\left(\frac{B_{2}}{B_{1}}\left(u+\pi|B_{1}|\right)\right)}.
\end{align*}
Let $\gamma\left(\frac{2B_{2}u}{B_{1}}\right)=1,$ then 
\begin{align}\label{P4Construction-WaveletEqn9}
\Gamma^{M_{1},M_{2}}\left(\frac{B_{2}u}{B_{1}}\right)=e^{-\frac{iu}{B_{1}}}\frac{\sqrt{2}}{2}\sum_{n\in\mathbb{Z}}(-1)^n\overline{c_{n}}e^{-\frac{i}{2}\left(\frac{n^2A_{1}}{4B_{1}}+\frac{4n}{B_{1}}-\frac{2nu}{B_{1}}\right)}.
\end{align}
From \eqref{P4Construction-WaveletEqn8} and \eqref{P4Construction-WaveletEqn9}, we get
$$\sum_{n\in\mathbb{Z}}d_{n}e^{\frac{i}{2}\left(\frac{n^2A_{1}}{4B_{1}}+\frac{4n}{B_{1}}-\frac{2nu}{B_{1}}\right)}=e^{-\frac{iu}{B_{1}}}\sum_{n\in\mathbb{Z}}(-1)^n\overline{c_{n}}e^{-\frac{i}{2}\left(\frac{n^2A_{1}}{4B_{1}}+\frac{4n}{B_{1}}-\frac{2nu}{B_{1}}\right)}.$$
Thus, we have
\begin{align*}
d_{k}e^{\frac{i}{2}\left(\frac{k^2A_{1}}{4B_{1}}+\frac{4k}{B_{1}}\right)}=\sum_{n\in\mathbb{Z}}(-1)^n\overline{c_{n}}e^{-\frac{i}{2}\left(\frac{n^2A_{1}}{4B_{1}}+\frac{4n}{B_{1}}\right)}\delta_{1-k-n}
\end{align*}
i.e.,
\begin{align*}
d_{k}e^{\frac{i}{2}\left(\frac{k^2A_{1}}{4B_{1}}+\frac{4k}{B_{1}}\right)}=(-1)^{1-k}\overline{c_{1-k}}e^{-\frac{i}{2}\left\{\frac{(1-k)^2A_{1}}{4B_{1}}+\frac{4(1-k)}{B_{1}}\right\}}.
\end{align*}
Thus, it follows that 
$$d_{k}=(-1)^{1-k}\overline{c_{1-k}}e^{-\frac{i}{2}\left\{\frac{(1-k)^2A_{1}}{4B_{1}}+\frac{k^2A_{1}}{4B_{1}}+\frac{4}{B_1}\right\}},~k\in\mathbb{Z}.$$

\section{Conclusions}
We have proposed a novel time-frequency analyzing tool, namely LCST, which combines the advantages of the LCT and the ST and offers time and linear canonical domain spectral information simultaneously in the time-LCT-frequency plane. The proposed transform is more general then that given in \cite{shah2020linear}. We see that our transform gives rise to some novel integral transform like novel FrST and novel Fresnel transform. We have studied its properties like inner product relation, reconstruction formula and also characterized its range via which we have obtained that the range of the transform is a reproducing kernel Hilbert space. We have also developed an MRA associated with the transform and discussed the construction of orthonormal basis of $L^2(\mathbb{R})$.  
\bibliography{P4MasterB4}
\bibliographystyle{plain}
\end{document}